\documentclass[a4paper, 12pt]{article}
\usepackage{amsmath}
\usepackage{amsthm}
\usepackage{amssymb}
\usepackage{url}
\usepackage{mathrsfs}
\usepackage[dvipdfmx]{graphicx, xcolor}
\usepackage{tikz}
\usetikzlibrary{positioning}
\usetikzlibrary{intersections}

\theoremstyle{definition}

\newtheorem{theorem}{Theorem}[section]
\newtheorem{definition}[theorem]{Definition}
\newtheorem{proposition}[theorem]{Proposition}
\newtheorem{lemma}[theorem]{Lemma}
\newtheorem{corollary}[theorem]{Corollary}

\newtheorem*{remark}{Remark}
\newtheorem*{acknowledgment}{Acknowledgments}

\newcommand{\Natural}{\mathbb{N}}

\newcommand{\Nonnegative}{\mathbb{N}_0}
\newcommand{\Nonnegativeinfty}{\mathbb{N}_0 \cup \left\{ \infty \right\}}
\newcommand{\Integer}{\mathbb{Z}}

\newcommand{\Real}{\mathbb{R}}
\newcommand{\Complex}{\mathbb{C}}
\newcommand{\abs}[1]{\left\lvert #1 \right\rvert}
\newcommand{\norm}[1]{\left\lVert #1 \right\rVert}

\newcommand{\set}[1]{\left\{ #1 \right\}}
\newcommand{\tset}[1]{\{ #1 \}}
\newcommand{\setcond}[2]{\left\{ #1 \,;\, #2 \right\}}

\newcommand{\dsum}[2]{\displaystyle{\sum_{#1}^{#2}}}

\newcommand{\T}{\mathcal{T}}

\newcommand{\Cayley}[2]{\mathrm{Cay} \left( #1 , #2 \right)}

\newcommand{\rad}[1]{\mathrm{rad} \left( #1 \right)}
\newcommand{\diam}[1]{\mathrm{diam} \left( #1 \right)}
\newcommand{\dist}[2]{d \left( #1 , #2 \right)}
\newcommand{\supp}[1]{\mathrm{supp} \left( #1 \right)}

\newcommand{\Matsp}[2]{\mathrm{Mat}_{#1} \left( #2 \right)}
\newcommand{\length}[1]{\ell \left( #1 \right)}

\renewcommand{\L}{\mathcal{L}}
\newcommand{\R}{\mathcal{R}}

\numberwithin{equation}{section}

\title{\textbf{Hypergroups derived from random walks on some infinite graphs}}
\author{Tomohiro IKKAI\footnote{e-mail: m13006z@math.nagoya-u.ac.jp}, Yusuke SAWADA\footnote{e-mail: m14017c@math.nagoya-u.ac.jp}}
\date{}

\begin{document}

\maketitle
\begin{center}
Graduate School of Mathematics, 

Nagoya University, 

Nagoya, 464-8602, JAPAN
\end{center}

\renewcommand{\thefootnote}{}



\renewcommand{\thefootnote}{\arabic{footnote}}
\setcounter{footnote}{0}


\begin{abstract}
Wildberger gave a method to construct a finite hermitian discrete hypergroup from a random walk on a certain kind of finite graphs. 
In this article, we reveal that his method is applicable to a random walk on a certain kind of infinite graphs. 
Moreover, we make some observations of finite or infinite graphs on which a random walk produces a hermitian discrete hypergroup. 
\end{abstract}


\section{Introduction} \label{SecIntro}

Let $\Natural$ be the set of positive integers, 
$\Nonnegative = \Natural \cup \tset{0}$, 
$\Integer$ be the ring of rational integers, 
$\Real$ be the field of real numbers, 
$\Real_{+}$ be the set of non-negative real numbers and $\Complex$ be the field of complex numbers, respectively.

The concept of hypergroups is a probability theoretic extension of the concept of locally compact groups. 
It was introduced by Dunkl \cite{Dun73}, Jewett \cite{Jew75} and Spector \cite{Spe78}, 
and many authors developed harmonic analysis and representation theory on hypergroups as well as those on locally compact groups. 

Discrete hypergroups, including finite hypergroups, are the main objects in this article. 
We state the definition of discrete hypergroups in Section \ref{SecPreHG}, 
but omit the general definition of the non-discrete cases in this article. 
For the general definition of hypergroups and a fundamental theory of hypergroups, see \cite{Blo-Hey} for example. 
Actually, the concept of (finite) discrete hypergroups was implicitly utilized by Frobenius around 1900, before Dunkl, Jewett and Spector. 
Recently, it was shown by Wildberger \cite{Wil95} that finite commutative hypergroups have some connections 
with number theory, conformal field theory, subfactor theory and other fields. 

Structures of finite hypergroups of order two and three are completely determined. 
Indeed, those of order two are classically known, and those of order three were investigated by Wildberger \cite{Wil02}. 
However, few examples of finite hypergroups of order four or greater were known; 
A recent result by Matsuzawa, Ohno, Suzuki, Tsurii and Yamanaka \cite{MOSTY} gives 
some examples of non-commutative hypergroups of order five. 
Wildberger \cite{Wil94}, \cite{Wil95} gave a way to construct a finite commutative hypergroup 
from a random walk on a certain kind of finite graphs. 
By his method, we have some examples of hypergroups of large order. 
In particular, the authors disclose that random walks on prism graphs produce finite commutative hypergroups of arbitrarily large order. 
After recalling Wildberger's method in Section \ref{SecDistReg}, 
we will give the structures of those hypergroups derived from prism graphs in Section \ref{SecHGPGraph}. 

It is known that not all finite graphs produce hypergroups. 
Wildberger mentioned 
that a random walk on any strong regular graph and on any distance transitive graph produces a hypergroup of order three 
and suggested that a random walk on any distance-regular graphs produces a finite hypergroup. 
(We do not discuss on distance transitive graphs in this article, but it is known that distance transitive graphs are always distance-regular.) 
The authors verify that a random walk on a distance-regular graph certainly produces a finite hypergroup. 
This claim will be restated as Theorem \ref{ThmDRGHG} in Section \ref{SecDRGWil}. 

Distance-regular graphs can be defined by using terms of association schemes. 
Wildberger \cite{Wil95} gave a way to construct a finite hypergroup from the Bose-Mesner algebra of an association scheme 
without using any graph theoretical terms, so that we have two ways to construct hypergroups from a distance-regular graph. 
The authors also verify that these two ways produce the same hypergroup, which is proved in Section \ref{SecDRGWil}. 

The authors discovered that Wildberger's method also works for some infinite graphs, 
and a random walk on such a graph produces a countably infinite but discrete hypergroup. 
This article is the first attempt to construct infinite hypergroups from random walks on graphs. 
In Section \ref{SecDRGEx} and Section \ref{SecHGPGraph}, 
we give several examples of infinite regular graphs on which a random walk produces a discrete hypergroup. 
The following theorem is one of the main results 
and gives examples of infinite regular graphs on which a random walk produces a discrete hypergroup. 

\begin{theorem} \label{ThmInfty}
$\mathrm{(a)}$ \ For any $k \in \Natural$ with $k \geq 2$, a random walk on the infinite $k$-regular tree produces a discrete hypergroup. 

$\mathrm{(b)}$ \ A random walk on the Cayley graph 
$\Cayley{\Integer \oplus \left( \Integer / 2 \Integer \right)}{\set{(\pm 1, \overline{0}), (0, \overline{1})}}$ produces a discrete hypergroup. 
(The symbols $\overline{0}$ and $\overline{1}$ denote the residue classes of $0$ and $1$ modulo $2$, respectively.)
\end{theorem}

The definitions of Cayley graphs and the infinite $k$-regular tree will be given in Section \ref{SecDRGEx}. 

It is a graph theoretical problem to decide whether a random walk on a given graph produces a discrete hypergroup. 
The purpose of this article is to formulate Wildberger's construction as a graph theoretical problem in the sense of the above and 
to establish a fundamental theory on that problem. 

Contents of this article are as follows: 
Section \ref{SecPrelim} of this article is devoted 
to definitions, notations and basic facts on discrete hypergroups, graphs and association schemes. 

In Section \ref{SecDistReg}, we recall Wildberger's construction of a discrete hypergroup from a random walk on a graph 
in an extended form including the case when the graph is infinite. 
Section \ref{SecDRGWil} gathers some fundamental propositions to formulate our problems 
and provides a proof that a random walk on a distance-regular graph always produces a discrete hypergroup. 
Section \ref{SecDRGEx} furnishes two kinds of examples of finite distance-regular graphs (complete graphs and platonic solids) 
and two kinds of examples of infinite ones (infinite regular trees and the ``linked-triangle graph''). 
The distance-regularity of finite graphs is a classical concept, 
and it is already known that random walks on complete graphs and platonic solids produce discrete hypergroups.
However, the distance-regularity of infinite graphs has not been discussed even in graph theory. 
The authors proved that the two kinds of infinite graphs are certainly distance-regular 
and the claim $\mathrm{(a)}$ of Theorem \ref{ThmInfty}, and proofs of them will be given in that section. 
In particular, the ``linked-triangle graph,'' constructed in that section, 
was an unexpected example on which a random walk produce a discrete hypergroup. 

We will consider non-distance-regular graphs in Section \ref{SecNDR}. 
In Section \ref{SecNDRGen}, further observations than those in Section \ref{SecDRGWil} will be made. 
One can meet some examples of non-distance-regular graphs 
(prisms, complete bipartite graphs, the infinite ladder graph and some intriguing finite graphs) in Section \ref{SecHGPGraph}. 
These are new examples on which random walks produce discrete hypergroups. 
We will mention the essence of a proof of Theorem \ref{ThmInfty} $\mathrm{(b)}$ in the same section. 


\section{Preliminaries} \label{SecPrelim}

As preliminaries to arguments in the following sections, 
we will give some definitions, notations and fundamental propositions from hypergroup theory and graph theory in this section. 

First, we provide some notations used in the following arguments. 

\begin{itemize}
\item For a set $X$, we let $\abs{X}$ denote the number of elements in $X$. 
For an infinite set $X$, we just say $\abs{X} = \infty$ and do not consider its cardinality. 

\item We are to use the convention that $\infty > n$ for any $n \in \Nonnegative$. 



\item Let $X$ be an arbitrary set and $\mathcal{A}$ be a commutative ring. 
A map $\varphi : X \times X \rightarrow \mathcal{A}$ is identified 
with a $\mathcal{A}$-valued $X \times X$ matrix $A_{\varphi} = \left( \varphi(x, y) \right)_{x, y \in X}$. 
The set of $\mathcal{A}$-valued $X \times X$ matrices is denoted by $\Matsp{X}{\mathcal{A}}$. 

The sum and the product of two matrices $A = (A_{x,y})_{x,y \in X}$, $B = (B_{x,y})_{x,y \in X} \in \Matsp{X}{\mathcal{A}}$ is defined 
in a usual way like 
\begin{equation*}
A + B = (A_{x,y} + B_{x,y})_{x,y \in X}. 
\end{equation*}
On the other hand, note that the product of two matrices $A = (A_{x,y})_{x,y \in X}$, $B = (B_{x,y})_{x,y \in X} \in \Matsp{X}{\mathcal{A}}$ 
can be defined in a usual way like 
\begin{equation*}
AB = \left( \sum_{z \in X} A_{x,z} B_{z,y} \right)_{x,y \in X}, 
\end{equation*}
only if either $A$ has finitely many non-zero entries in each rows and each columns or $B$ does. 
In particular, when $A$ has finitely many non-zero entries in each rows and each columns, we can define the $n$-th power $A^{n}$ of $A$ with $n \in \Nonnegative$. 
(As usual, $A^{0}$ is defined as the identity matrix $E = (\delta_{x,y})_{x,y \in X}$, where $\delta_{x,y}$ is the Kronecker delta.)
\end{itemize}


\subsection{Preliminaries from Hypergroup Theory} \label{SecPreHG}

We treat only ``discrete hypergroups'' in this article. 
In general, hypergroups are equipped with some topologies and require several complicated conditions. 
Such conditions for hypergroups with the discrete topologies can be simplified. 
For a general definition, see \cite{Blo-Hey} for example. 

Let $K$ be an arbitrary set and consider the free vector space 
\begin{equation*}
\Complex K = \bigoplus_{x \in K} \Complex x 
= \setcond{\dsum{i=1}{n} a_i x_i}{n \in \Natural, a_1, \cdots, a_n \in \Complex, x_1, \cdots, x_n \in K}, 
\end{equation*}
generated by $K$. 
We note that the $\ell^1$-norm $\norm{\cdot}_1 : \Complex K \rightarrow \Real_{+}$ such that 
\begin{equation*}
\norm{\mu}_1 = \sum_{x \in K} \abs{\mu (x)} \quad \left( \mu \in \Complex K \right), 
\end{equation*}
where $\mu (x)$ denotes the coefficient of $x$ as $\mu$ is expressed by a linear combination of elements of $K$, 
induces a normed space structure into $\Complex K$. 
For $\mu \in \Complex K$, we define the \textit{support} of $\mu$, denoted by $\supp{\mu}$, as 
\begin{equation*}
\supp{\mu} = \setcond{x \in K}{\mu(x) \neq 0}. 
\end{equation*}



Here is the definition of discrete hypergroups. 

\begin{definition} \label{DefHG}
We call a set $K$ a \textit{discrete hypergroup} if the following four conditions are fulfilled. 

\begin{itemize}
\item[$\mathrm{(i)}$] A binary operation 
\begin{equation*}
\circ : \Complex K \times \Complex K \ni (\mu, \nu) \mapsto \mu \circ \nu \in \Complex K, 
\end{equation*}
called a \textit{convolution}, and an involution map 
\begin{equation*}
* : \Complex K \ni \mu \mapsto \mu^{*} \in \Complex K
\end{equation*}
are defined, 
and $\left( \Complex K, \circ, * \right)$ form an associative $*$-algebra (i.e.\ $\Complex$-algebra equipped with an involution map) 
with the neutral element $e \in K$. 
(For the definition of an involution map, see \cite[Chapter 11]{Rud} for example.) 

\item[$\mathrm{(ii)}$] The convolution $x \circ y$ of two elements $x$, $y \in K$ satisfies the following conditions 
$\mathrm{(iia)}$ and $\mathrm{(iib)}$: 

\begin{itemize}
\item[$\mathrm{(iia)}$] For all $x$, $y$, $z \in K$, $(x \circ y) (z) \in \Real_{+}$. 

\item[$\mathrm{(iib)}$] For all $x$, $y \in K$, $\norm{x \circ y}_1 = 1$. 

\end{itemize}

\item[$\mathrm{(iii)}$] The involution $*$ maps $K$ onto $K$ itself. 

\item[$\mathrm{(iv)}$] For $x$, $y \in K$, the neutral element $e$ belongs to $\supp{x \circ y}$ if and only if $y = x^{*}$. 
\end{itemize}

A discrete hypergroup $K$ is said to be \textit{finite} if $K$ is a finite set, 
to be \textit{commutative} if $(\Complex K, \circ)$ is a commutative algebra 
and to be \textit{hermitian} if the restriction of the involution $* \mid_K$ is the identity map on $K$. 
\end{definition}

It can be easily verified that a hermitian discrete hypergroup must be commutative. 


The structure of a discrete hypergroup $K$ can be determined by results of the operations $\circ$ and $*$ for elements of $K$. 
In particular, under the assumption that $K$ is a hermitian discrete hypergroup, 
we can compute a convolution of arbitrary two elements in $\Complex K$ 
if we know the results of computations of $x \circ y$ for all $x$, $y \in K$ since the convolution $\circ$ is bilinear. 
In other words, if the convolution of two elements $x$, $y \in K$ is expressed as 
\begin{equation} \label{eqStructure}
x \circ y = \sum_{z \in K} P_{x, y}^{z} z 
\end{equation}
with some $P_{x,y}^{z} \in \Real_{+}$, 
we can express the convolution of any two elements of $\Complex K$ as a linear combination of elements of $K$ 
by using the coefficients $P_{x, y}^{z}$. 
In this article, the identities such as \eqref{eqStructure}, defining convolutions of two elements of $K$, 
are called the ``structure identities'' of $K$. 

\begin{remark}
A discrete group $G$ can be regarded as a discrete hypergroup. 
More precisely, defining a convolution on $\Complex G$ as the bilinear extension of the multiplication $G \times G \ni (x, y) \mapsto x y \in G$ 
and the involution on $\Complex G$ as the conjugate-linear extension of the inversion $G \ni x \mapsto x^{-1} \in G$, 
one can check that $G$ satisfies the conditions $\mathrm{(i)}$--$\mathrm{(iv)}$ in Definition \ref{DefHG}. 
These operations are the same as those of the group algebra $\Complex [G]$. 
This is the reason why the concept of discrete hypergroups is a generalization of the concept of discrete groups. 
\end{remark}


\subsection{Preliminaries from Graph Theory} \label{SecPreGraph}

In our context, all graphs are supposed to have neither any loops nor any multiple edges (i.e.\ to be simple graphs). 
We consider only \textit{connected} graphs, in which, for any two vertices $v$, $w$, there exists a path from $v$ to $w$, 
but we will handle both finite graphs and infinite graphs. 

Let $X = (V, E)$ be a graph with the vertex set $V$ and the edge set $E$. 
For two vertices $v$, $w \in V$, 
the \textit{distance} $\dist{v}{w}$ between $v$ and $w$ is defined as the length of the shortest paths from $v$ to $w$. 
Note that the distance function $d : V \times V \rightarrow \Nonnegative$ allows $V$ to be a metric space, 
that is, the function $\dist{\cdot}{\cdot}$ satisfies the following three properties: 
\begin{gather}
\dist{v}{w} = 0 \Longleftrightarrow v = w \quad (v, w \in V), \label{eqIndis} \\
\dist{v}{w} = \dist{w}{v} \quad (v, w \in V), \label{eqSym} \\ 
\dist{v}{w} \leq \dist{v}{u} + \dist{u}{w} \quad (u, v, w \in V). \label{eqTriIneq} 
\end{gather}
It is clear that two vertices $v$ and $w$ are adjacent if and only if $\dist{v}{w} = 1$. 
A path in $X$ is called a \textit{geodesic} if its length equals to the distance between the initial vertex and the terminal vertex. 
The \textit{eccentricity} $e(v) \in \Nonnegativeinfty$ of $v \in V$ is defined as 
\begin{equation} \label{eqEcc}
e(v) = \sup \setcond{\dist{v}{w}}{w \in V}. 
\end{equation}
If the set $\setcond{\dist{v}{w}}{w \in V}$ is unbounded above, then we define $e(v) = \infty$. 
This happens only if $X$ is infinite. 
The \textit{radius} $\rad{X} \in \Nonnegativeinfty$ and the \textit{diameter} $\diam{X} \in \Nonnegativeinfty$ of $X$ are defined as 
\begin{gather*}
\rad{X} = \min \setcond{e(v)}{v \in V}, \\
\diam{X} = \max \setcond{e(v)}{v \in V}. 
\end{gather*}

Clearly, that $\rad{X} = \infty$ implies that $\diam{X} = \infty$. 
A \textit{complete graph}, in which distinct vertices are mutually adjacent, is of diameter one, 
and a non-complete graph is of diameter two or greater. 
We say $X$ to be \textit{self-centered} if $X$ is finite and every $v \in V$ satisfies that $e(v) = \rad{X}$. 
For each $i \in \Nonnegative$ and $v \in V$, let $\Gamma_{i}(v)$ denote the set 
\begin{equation*}
\Gamma_{i}(v) = \setcond{w \in V}{\dist{v}{w} = i}. 
\end{equation*}
It immediately follows that $\Gamma_{i} (v) = \varnothing$ if $i > \diam{X}$.  
The number of elements in $\Gamma_{1} (v)$ is called the \textit{degree} of $v$. 
The graph $X$ is said to be \textit{$k$-regular} for some $k \in \Natural$ if every vertex $v \in V$ is of degree $k$. 
We are going to deal with graphs of which every vertex has a finite degree in this article. 


An \textit{automorphism} of a graph $X = (V, E)$ is a bijection $\varphi : V \rightarrow V$ such that, for two vertices $v$, $w \in V$, 
$\varphi(v)$ and $\varphi(w)$ are adjacent in $X$ if and only if $v$ and $w$ are adjacent in $X$.  
A graph $X$ is said to be \textit{vertex-transitive} if, for any two vertices $v$, $w \in V$, 
there exists an automorphism $\varphi$ of $X$ such that $\varphi(v) = w$. 


The ``distance-regular graphs'' play a key role in the following part. 
Here we recall the definition of such graphs, and some examples will appear in Section \ref{SecDRGEx}. 

\begin{definition} \label{DefDRG}
A (finite or infinite) connected graph $X = (V, E)$ is said to be \textit{distance-regular} if the following three conditions are fulfilled. 
\begin{itemize}
\item[$\mathrm{(i)}$] There exists $b_0 \in \Natural$ such that $\abs{\Gamma_{1} (v)} = b_0$ for any $v \in V$. 

\item[$\mathrm{(ii)}$] For each integer $i$ with $1 \leq i < \diam{X}$, there exist $b_i \in \Natural$ and $c_i \in \Natural$ such that 
$\abs{\Gamma_{i+1}(v) \cap \Gamma_{1}(w)} = b_i$ and $\abs{\Gamma_{i-1}(v) \cap \Gamma_{1}(w)} = c_i$ 
for any $v$, $w \in V$ with $\dist{v}{w} = i$. 

\item[$\mathrm{(iii)}$] If $s = \diam{X} < \infty$, there exists $c_s \in \Natural$ 
such that $\abs{\Gamma_{s-1}(v) \cap \Gamma_{1}(w)} = c_s$ for any $v$, $w \in V$ with $\dist{v}{w} = s$. 
\end{itemize}

For a distance-regular graph $X$, 
the sequence $(b_0, b_1, b_2, \cdots; c_1, c_2, c_3, \cdots)$ of integers is called the \textit{intersection array} of $X$. 
\end{definition}

Note that a distance-regular graph is $b_0$-regular. 

Distance-regular graphs are important in combinatorics: 
Its vertex set has a structure of an ``association scheme.'' 

\begin{definition} \label{DefAssSch}
Let $Y$ be a set and $\R = \set{R_{i}}_{i \in I}$ be a partition of $Y \times Y$ 
(i.e.\ $R_{i}$'s are mutually disjoint subsets of $Y \times Y$ and $Y \times Y = \bigcup_{i \in I} R_{i}$), 
which consists of countably many non-empty sets. 
The pair $(Y, \R)$ is called a (\textit{symmetric}) \textit{association scheme} if the following conditions are all satisfied. 

\begin{itemize}

\item[$\mathrm{(i)}$] The diagonal set $\setcond{(y, y)}{y \in Y}$ of $Y \times Y$ is an entry of $\R$. 

\item[$\mathrm{(ii)}$] If $(x, y) \in R_i$, then $(y, x) \in R_i$. 

\item[$\mathrm{(iii)}$] For any $i$, $j$, $k \in I$, 
there exists $p_{i,j}^{k} \in \Nonnegative$ such that 
\begin{equation*}
\abs{\setcond{z \in Y}{(x, z) \in R_i , (z, y) \in R_j}} = p_{i,j}^{k}
\end{equation*}
for each pair $(x, y) \in R_k$. 
\end{itemize}

The numbers $p_{i,j}^{k}$ in the condition $\mathrm{(iii)}$ are called the \textit{intersection numbers} of $(Y, \R)$. 
\end{definition}

Given a connected graph $X = (V, E)$, we can find a canonical partition of $V \times V$; 
Let 
\begin{gather}
I = 
 \begin{cases}
  \set{0, 1, \cdots, \diam{X}} & (\diam{X} \in \Natural), \\
  \Nonnegative & (\diam{X} = \infty), \\
 \end{cases} \label{eqPartition1} \\
R_i = \setcond{(v, w) \in V \times V}{\dist{v}{w} = i} \quad (i \in I) \label{eqPartition2}
\end{gather}
and $\R (X) = \set{R_i}_{i \in I}$. 
If every vertex of $X$ has a finite degree, then $R_i$'s are all non-empty sets and $\R (X)$ forms a partition of $V \times V$ 
(see Proposition \ref{PropWellDefined}). 
Furthermore, we find that, by \eqref{eqIndis}, 
\begin{equation*}
R_0 = \setcond{(v, v)}{v \in V}
\end{equation*}
and that, by \eqref{eqSym}, 
\begin{equation*}
(v, w) \in R_i \Longrightarrow (w, v) \in R_i 
\end{equation*}
for any $i \in I$. 

We prove that distance-regular graphs can be characterized by the condition $\mathrm{(iii)}$ in Definition \ref{DefAssSch}. 
This is a classical result in the finite graph case, 
so that a proof of the following proposition for the case when $X$ is a finite graph is introduced in \cite{Bai}. 
Even if $X$ is infinite, the same method as that on \cite{Bai} works. 

\begin{proposition} \label{PropDRGAS}
Let $X = (V, E)$ be a connected graph and $s = \diam{X} \in \Nonnegativeinfty$. 
Then the followings are equivalent. 

\begin{itemize}
\item[$\mathrm{(a)}$] The graph $X$ is distance-regular. 

\item[$\mathrm{(b)}$] The pair $\left(V, \R (X) \right)$ forms an association scheme. 
\end{itemize}
\end{proposition}

\begin{proof}
First, suppose that $X$ satisfies the condition $\mathrm{(b)}$. 
Then, for two vertices $v$, $w \in V$, that $\dist{v}{w} = k$ is equivalent to that $(v, w) \in R_k$, and we have 
\begin{equation*}
\Gamma_{j} (v) \cap \Gamma_{1} (w) = \setcond{u \in V}{(v, u) \in R_j , (u, w) \in R_1}. 
\end{equation*}
Hence we can find that $X$ is a distance-regular graph with the intersection array 
$(p_{1,1}^{0}, p_{2,1}^{1}, \cdots, p_{i,1}^{i-1}, \cdots; p_{0,1}^{1}, p_{1,1}^{2}, \cdots, p_{i-1,1}^{i}, \cdots)$ 
(the $i$-th entry of the former array is $p_{i,1}^{i-1}$, and that of the latter array is $p_{i-1,1}^{i}$). 

Conversely, suppose that $X$ is a distance-regular graph with the intersection array $(b_0, b_1, b_2, \cdots; c_1, c_2, c_3, \cdots)$. 
Since the conditions $\mathrm{(i)}$ and $\mathrm{(ii)}$ in Definition \ref{DefAssSch} hold, 
we now check that the condition $\mathrm{(iii)}$ holds. 

Consider the family of matrices $\set{A^{(i)}}_{i \in I} \subset \Matsp{V}{\Integer}$ whose entries $A_{v,w}^{(i)}$ are defined as 
\begin{equation} \label{eqDistmat}
A_{v,w}^{(i)} = 
 \begin{cases}
  1 & \left( (v, w) \in R_i \right), \\
  0 & \left( (v, w) \notin R_i \right). 
 \end{cases}
\end{equation}
Then, we find that 
\begin{equation} \label{eqLinearize}
A^{(i)} A^{(1)} = b_{i-1} A^{(i-1)} + (b_0 - b_i - c_i) A^{(i)} + c_{i+1} A^{(i+1)}
\end{equation}
for any $i \in I$ since 
\begin{align*}
\left( A^{(i)} A^{(1)} \right)_{v,w} =& \abs{\setcond{u \in V}{(v, u) \in R_i , (u, w) \in R_1}} \\
=& \abs{\Gamma_{i} (v) \cap \Gamma_{1} (w)} \\
=& \begin{cases}
       b_{i-1} & \left( (v, w) \in R_{i-1} \right), \\
       b_0 - b_i - c_i & \left( (v, w) \in R_{i} \right), \\
       c_{i+1} & \left( (v, w) \in R_{i+1} \right), \\
       0 & \left( \text{otherwise} \right).
     \end{cases}
\end{align*}
Here, we use the convention that $b_{-1} = b_{s} = 0$ and $c_{0} = c_{s+1} = 0$. 
The equation \eqref{eqLinearize} implies that $\left( A^{(1)} \right)^{n}$ 
can be written as a linear combination of $A^{(k)}$'s for any $n \in \Natural$. 

On the other hand, from \eqref{eqLinearize}, 
we can verify by induction on $i$ that each $A^{(i)}$ can be expressed as a polynomial of $A^{(1)}$. 
Therefore, it turns out that the product of two matrices $A^{(i)} A^{(j)}$ can be written as a linear combination of $A^{(k)}$'s. 
Moreover, the coefficients in the linear combination expressing $A^{(i)} A^{(j)}$ must be all non-negative integers. 

Let $p_{i,j}^{k} \in \Nonnegative$ denote the coefficient of $A^{(k)}$ in the linear combination expressing $A^{(i)} A^{(j)}$. 
(Note that $A^{(k)}$'s are linearly independent over $\Complex$.) 
Then, for any $(v, w) \in R_k$, we obtain that 
\begin{align}
&\abs{\setcond{u \in V}{(v, u) \in R_i , (u, w) \in R_j}} \notag \\
=& \left( A^{(i)} A^{(j)} \right)_{v,w} \notag \\
=& \sum_{l \in I} p_{i,j}^{l} A_{v,w}^{(l)} \notag \\
=& p_{i,j}^{k} \label{eqAssSch}
\end{align}
since 
\begin{equation*}
A_{v,w}^{(l)} = \begin{cases}
                 1 & (l = k), \\
                 0 & (l \neq k).
                \end{cases}
\end{equation*}
The calculations \eqref{eqAssSch} warrant that $\left(V, \R (X) \right)$ satisfies 
the condition $\mathrm{(iii)}$ in Definition \ref{DefAssSch} and that $\left(V, \R (X) \right)$ forms an association scheme. 
\end{proof}

In general, it is difficult to express $p_{i,j}^{k}$ in terms of entries $b_l$'s and $c_l$'s of the intersection array. 

Here are several useful formulas on intersection numbers $p_{i,j}^{k}$ of an association scheme. 
Those are known as classical results, and one can find the same formulas on some textbooks, \cite{Bro-Coh-Neu} etc. 
Nevertheless, the authors could not find any proofs of the following identities except for $\mathrm{(e)}$, 
which is written in \cite{Bro-Coh-Neu}. 
The proofs of the followings consist of only elementary calculus in combinatorics. 

\begin{proposition} \label{PropAssEq}
Let $\left( Y, \R = \set{R_i}_{i \in I} \right)$ be an association scheme and $R_0 \in \R$ be the diagonal set of $Y \times Y$. 
Then, for any $i$, $j$, $k$, $m \in I$, the following identities hold. 

\begin{itemize}
\item[$\mathrm{(a)}$] $p_{0,j}^{k} = \delta_{j,k}$. 
\item[$\mathrm{(b)}$] $p_{i,j}^{0} = \delta_{i,j} p_{j,j}^{0}$. 
\item[$\mathrm{(c)}$] $p_{i,j}^{k} = p_{j,i}^{k}$. 
\item[$\mathrm{(d)}$] $\sum_{j \in I} p_{i,j}^{k} = p_{i,i}^{0}$. 
\item[$\mathrm{(e)}$] $\sum_{l \in I} p_{i,j}^{l} p_{l,k}^{m} = \sum_{l \in I} p_{j,k}^{l} p_{i,l}^{m}$. 
\item[$\mathrm{(f)}$] $p_{i,j}^{k} p_{k,k}^{0} = p_{i,k}^{j} p_{j,j}^{0}$. 
\end{itemize} 
\end{proposition}

\begin{proof}
For the identity $\mathrm{(a)}$, take $(x, y) \in R_k$. 
Then, we have 
\begin{equation*}
p_{0,j}^{k} = \abs{\setcond{z \in Y}{(x, z) \in R_0, (z, y) \in R_j}}. 
\end{equation*}
The set of which we are counting elements is empty when $j \neq k$ and equals to $\set{x}$ when $j = k$, 
so that we obtain $\mathrm{(a)}$. 

For the identity $\mathrm{(b)}$, take $(x, x) \in R_0$. 
(The set $R_0$ is the diagonal set of $Y \times Y$.) 
Then, we have 
\begin{equation*}
p_{i,j}^{0} = \abs{\setcond{z \in Y}{(x, z) \in R_i, (z, x) \in R_j}}. 
\end{equation*}
The desired identity is derived from the condition $\mathrm{(iii)}$ in Definition \ref{DefAssSch}. 

The identity $\mathrm{(c)}$ is also derived from the condition $\mathrm{(iii)}$ in Definition \ref{DefAssSch}. 

For the identity $\mathrm{(d)}$, take $(x, y) \in R_k$. 
Since $\R$ is a partition of $Y \times Y$, we have 
\begin{align*}
&\sum_{j \in I} p_{i,j}^{k} \\
=& \sum_{j \in I} \abs{\setcond{z \in Y}{(x, z) \in R_i, (z, y) \in R_j}} \\
=& \abs{\setcond{z \in Y}{(x, z) \in R_i}} \\
=& p_{i,i}^{0}. 
\end{align*}

To verify the identity $\mathrm{(e)}$, 
take $(x, y) \in R_m$ and count elements in the set 
\begin{equation*}
S = \setcond{(z, u) \in Y \times Y}{(x, z) \in R_i, (z, u) \in R_j, (u, y) \in R_k}. 
\end{equation*}
One can see that 
\begin{equation*}
p_{i,j}^{l} p_{l,k}^{m} = \abs{\setcond{(z, u) \in Y \times Y}{(x, u) \in R_l, (u, y) \in R_k, (x, z) \in R_i, (z, u) \in R_j}}, 
\end{equation*}
and it follows that $\sum_{l \in I} p_{i,j}^{l} p_{l,k}^{m} = \abs{S}$. 
On the other hand, one find that 
\begin{equation*}
p_{j,k}^{l} p_{i,l}^{m} = \abs{\setcond{(z, u) \in Y \times Y}{(x, z) \in R_i, (z, y) \in R_l, (z, u) \in R_j, (u, y) \in R_k}}
\end{equation*}
holds, and it follows that $\sum_{l \in I} p_{j,k}^{l} p_{i,l}^{m} = \abs{S}$. 
Combining these results, we obtain the desired identity. 

We can substitute $m = 0$ into the identity $\mathrm{(e)}$ and use $\mathrm{(b)}$ and $\mathrm{(c)}$ to get the identity $\mathrm{(f)}$. 
\end{proof}

Finally, we recall the definition of ``strongly regular graphs,'' a special class of distance-regular graphs. 

\begin{definition} \label{DefSRG}
A \textit{strongly regular graph} is a distance-regular graph with diameter two. 
\end{definition}

Strongly regular graphs can be characterized by four parameters $(n, k, \lambda, \mu)$, 
where $n = \abs{V}$, $k = b_0 = p_{1,1}^{0}$, $\lambda = p_{1,1}^{1}$ and $\mu = c_2 = p_{1,1}^{2}$. 
Hence the strongly regular graph with parameters $(n, k, \lambda, \mu)$ is called ``$(n, k, \lambda, \mu)$-strongly regular graph.'' 
(In this article, we regard complete graphs as not strongly regular since the parameter $\mu$ is not well-defined for complete graphs.) 


\section{Hermitian discrete hypergroups derived from distance-regular graphs} \label{SecDistReg}

\subsection{Wildberger's construction} \label{SecDRGWil}

Wildberger \cite{Wil95} gave a way to construct finite hermitian discrete hypergroups from a certain class of finite graphs. 
Coefficients appearing in the structure identities of such a discrete hypergroup are given 
by certain probabilities coming from a random walk on a corresponding graph. 
Actually, his method is applicable to some infinite graphs, so that we formulate his method in a general form. 
Here, we recall Wildberger's construction of a hermitian discrete hypergroup from a random walk on a graph. 

Given a finite or infinite connected graph $X = (V, E)$, whose vertices all have finite degree, 
and given a vertex $v_0 \in V$ as a ``base point,'' 
we consider the canonical partition $\R (X) = \set{R_i}_{i \in I}$ of $V \times V$ defined as \eqref{eqPartition1} -- \eqref{eqPartition2}. 

Now, we define a convolution $\circ = \circ_{v_0} : \Complex \R (X) \times \Complex \R (X) \rightarrow \Complex \R (X)$ 
for the base point $v_0$. 
We let $P_{i,j}^{k}$ for $i$, $j$, $k \in I$ denote the following probability: 
Consider a `jump' to a random vertex $w \in \Gamma_{j} (v)$ after a `jump' to a random vertex $v \in \Gamma_{i} (v_0)$. 
Let $P_{i,j}^{k}$ denote the probability that $w$ belongs to $\Gamma_{k} (v_0)$. 
These probabilities $P_{i,j}^{k}$ are explicitly given by 
\begin{equation} \label{eqDefP}
P_{i,j}^{k} = 
\frac{1}{\abs{\Gamma_{i} (v_0)}} \sum_{v \in \Gamma_{i} (v_0)} \frac{\abs{\Gamma_{j} (v) \cap \Gamma_{k} (v_0)}}{\abs{\Gamma_{j} (v)}}. 
\end{equation}
In general, the probabilities $P_{i,j}^{k}$ depend on the choice of the base point $v_0$. 

We note that the denominator $\abs{\Gamma_{j} (v)}$ can be zero so that $P_{i,j}^{k}$'s are not necessarily well-defined, 
but we can determine when $P_{i,j}^{k}$ are well-defined. 

\begin{proposition} \label{PropWellDefined}
Let $X = (V, E)$ be a connected graph whose vertices all have finite degrees and the base point $v_0 \in V$ be arbitrarily given. 
\begin{itemize}
\item[$\mathrm{(a)}$] When $X$ is infinite, $P_{i,j}^{k}$ are all well-defined. 
\item[$\mathrm{(b)}$] Suppose that $X$ is finite. 
Then, $P_{i,j}^{k}$ are all well-defined if and only if $X$ is a self-centered graph. 
\end{itemize}
\end{proposition}

\begin{proof}
First, we suppose that $X$ is infinite. 
It suffices to show that $\Gamma_{j} (v) \neq \varnothing$ for any $v \in V$ and any $j \in I$. 

If $\Gamma_{j} (v) = \varnothing$ for some $v \in V$ and some $j \in I$, 
we would find that $\Gamma_{j'} (v) = \varnothing$ for any $j' \in I$ such that $j' \geq j$. 
Indeed, a geodesic $v = v_0 \rightarrow v_1 \rightarrow \cdots \rightarrow v_{j'-1} \to v_{j'} = w$ from $v$ to $w \in \Gamma_{j'} (v)$ should pass through a vertex $v_j \in \Gamma_{j} (v)$ when $j' \geq j$.  
Then, we would have $V = \bigcup_{l=0}^{j-1} \Gamma_{l} (v)$. 

On the other hand, by our assumption that all vertices of $X$ have finite degrees, 
we could find that $\Gamma_{l} (v)$ must be a finite set for every $l \in \set{0, 1, \cdots, j-1}$. 
We had that the vertex set $V$, a union of finitely many finite sets, should be a finite set, 
but this is a contradiction to the assumption that $V$ is an infinite set, so that the claim $\mathrm{(a)}$ of the proposition has been proved. 

Next, we suppose that $X$ is finite and $P_{i,j}^{k}$ are all well-defined. 
We will show that $e(v) = \diam{X}$ for every $v \in V$. 

If there would exist a vertex $v \in V$ such that $e(v) < \diam{X}$, we had $\Gamma_{\diam{X}} (v) = \varnothing$. 
Then, setting $\dist{v_0}{v} = i$ and $\diam{X} = s$, we would find that 
\begin{equation*}
P_{i,s}^{k} = 
\frac{1}{\abs{\Gamma_{i}(v_0)}} \sum_{w \in \Gamma_{i}(v_0)} 
\frac{\abs{\Gamma_{s}(w) \cap \Gamma_{k}(v_0)}}{\abs{\Gamma_{s}(w)}} 
\end{equation*}
cannot be well-defined since $v \in \Gamma_{i}(v_0)$ and $\abs{\Gamma_{s} (v)} = 0$. 
Accordingly, we find that $e(v) = \diam{X}$ for every $v \in V$. 
This means that all of $e(v)$ coincide one another and we obtain that $e(v) = \rad{X}$ for every $v \in V$. 

Conversely, suppose that $X$ is finite and self-centered. 
A similar argument to the above one yields that $e(v) = \diam{X}$ for every $v \in V$. 
Then, we have that $\Gamma_{j} (v) \neq \varnothing$. 
Indeed, a geodesic starting from $v \in V$ of the length equal to $\diam{X}$ passes 
through a vertex belonging to $\Gamma_{j}(v)$ for each $j \in I$. 
This implies that the denominators of fractions in \eqref{eqDefP} never vanish. 
\end{proof}

When $P_{i,j}^{k}$ are all well-defined, we define a bilinear convolution $\circ = \circ_{v_0}$ on $\Complex \R (X)$ by the following identities: 
\begin{equation*}
R_i \circ_{v_0} R_j = \sum_{k \in I} P_{i,j}^{k} R_k \quad (i, j \in I).
\end{equation*}
Equipped with a convolution, $\Complex \R (X)$ becomes a unital $\Complex$-algebra which is not necessarily associative or commutative. 
(The existence of the neutral element will be shown below.)
In addition, we define a conjugate-linear map $* : \Complex \R (X) \rightarrow \Complex \R (X)$ by $* \mid_{\R (X)} = \mathrm{id}_{\R (X)}$, 
which is expected to be an involution on $\Complex \R (X)$. 
This map $*$ becomes an involution on $\Complex \R (X)$ if and only if $\Complex \R (X)$ is commutative. 

It can be easily verified that 
\begin{equation*}
P_{i,0}^{k} = \delta_{i,k}
=\begin{cases}
   1 & (k = i), \\
   0 & (k \neq i), \\
  \end{cases}
\quad
P_{0,j}^{k} = \delta_{j,k}
=\begin{cases}
   1 & (k = j), \\
   0 & (k \neq j), \\
  \end{cases}
\end{equation*}
so that $R_0 \in \R(X)$ should be the neutral element with respect to the convolution $\circ_{v_0}$ on $\Complex \R (X)$. 
One can also verify the followings.

\begin{proposition} \label{PropHg}
Let $X = (V, E)$ be a connected graph whose vertices all have finite degrees and the base point $v_0 \in V$ be arbitrarily given. 
Suppose that $P_{i,j}^{k}$ are well-defined for all $i$, $j$, $k \in I$. 
\begin{itemize}
\item[$\mathrm{(a)}$] For any $i$, $j \in I$, we have that $R_i \circ_{v_0} R_j \in \Complex \R (X)$ satisfies the followings: 
\begin{gather}
(R_i \circ_{v_0} R_j) (R_k) \in \Real_{+} \ \text{for all}\ k \in I, \label{eqProbmeas1} \\
\norm{R_i \circ_{v_0} R_j}_1 = 1, \label{eqProbmeas2} \\
\supp{R_i \circ_{v_0} R_j} \subset \set{\abs{i-j}, \abs{i-j}+1, \cdots, i+j}. \label{eqProbmeas3} 
\end{gather}
\item[$\mathrm{(b)}$] For $i$, $j \in I$, the neutral element $R_0$ belongs to $\supp{R_i \circ_{v_0} R_j}$ if and only if $i = j$. 
\end{itemize}
\end{proposition}

\begin{proof}
The first assertion \eqref{eqProbmeas1} immediately follows from \eqref{eqDefP}, the definition of $P_{i,j}^{k}$. 

Since $\setcond{\Gamma_{j} (v) \cap \Gamma_{k} (v_0)}{k \in I}$ forms a partition of $\Gamma_{j} (v)$, 
we can compute $\norm{R_i \circ_{v_0} R_j}_1$ to be 
\begin{align*}
\norm{R_i \circ_{v_0} R_j}_1 =& \sum_{k \in I} P_{i,j}^{k} \\
=& \sum_{k \in I} \frac{1}{\abs{\Gamma_{i} (v_0)}} 
     \sum_{v \in \Gamma_{i} (v_0)} \frac{\abs{\Gamma_{j} (v) \cap \Gamma_{k} (v_0)}}{\abs{\Gamma_{j} (v)}} \\
=& \sum_{v \in \Gamma_{i} (v_0)} \frac{1}{\abs{\Gamma_{i} (v_0)} \cdot \abs{\Gamma_{j} (v)}}
     \sum_{k \in I} \abs{\Gamma_{j} (v) \cap \Gamma_{k} (v_0)} \\
=& 1. 
\end{align*}
When $P_{i,j}^{k} > 0$, we can find a vertex $v \in \Gamma_{i} (v_0)$ such that $\Gamma_{j} (v) \cap \Gamma_{k} (v_0) \neq \varnothing$. 
Then, any $w \in \Gamma_{j} (v) \cap \Gamma_{k} (v_0)$ should satisfy that 
\begin{align*}
\abs{i - j} &= \abs{\dist{v_0}{v} - \dist{w}{v}} \\
&\leq k = \dist{v_0}{w} \\
&\leq \dist{v_0}{v} + \dist{v}{w} = i + j. 
\end{align*}
This means that $\supp{R_i \circ_{v_0} R_j} \subset \set{\abs{i - j}, \abs{i - j} + 1, \cdots, i+j}$. 
We have shown the claim $\mathrm{(a)}$ in the above arguments.  

Now, we prove the claim $\mathrm{(b)}$. 
If $i = j$, the base point $v_0$ belongs to $\Gamma_{i} (v) \cap \Gamma_{0} (v_0)$ for every $v \in \Gamma_{i} (v_0)$, 
so we have that $P_{i,i}^{0} > 0$. 
Next, we suppose that $R_0 \in \supp{R_i \circ_{v_0} R_j}$. 
Then, $P_{i,j}^{0} > 0$ holds, 
and this yields that there exists a vertex $v \in \Gamma_{i} (v_0)$ such that $\Gamma_{j} (v) \cap \Gamma_{0} (v_0) \neq \varnothing$. 
This causes $v_0$ to lie in $\Gamma_{j} (v)$. 
Therefore, we have that $i = \dist{v}{v_0} = j$. 
\end{proof}

Appealing to Proposition \ref{PropHg}, 
one finds that $\R(X)$ becomes a hermitian hypergroup if and only if the convolution $\circ_{v_0}$ is associative and commutative. 
The convolution $\circ_{v_0}$ is not always associative or commutative, 
and it is a problem when $\circ_{v_0}$ is both associative and commutative. 

An answer to this problem was given by Wildberger \cite{Wil95}; 
If $X$ is a strongly regular graph, then $\R (X)$ becomes a hermitian discrete hypergroup. 
Let $X = (V, E)$ be an $(n, k, \lambda, \mu)$-strongly regular graph. 
Then, structure identities of $\R(X) = \set{R_0, R_1, R_2}$ are given by 
\begin{gather}
R_1 \circ R_1 = \frac{1}{k} R_0 + \frac{\lambda}{k} R_1 + \frac{k-\lambda-1}{k} R_2, \label{eqSreg1} \\
R_1 \circ R_2 = R_2 \circ R_1 = \frac{\mu}{k} R_1 + \frac{k-\mu}{k} R_2, \label{eqSreg2} \\
R_2 \circ R_2 = \frac{1}{n-k-1} R_0 + \frac{k-\mu}{n-k-1} R_1 + \frac{n+\mu-2k-2}{n-k-1} R_2, \label{eqSreg3}
\end{gather}
which are independent of a choice of the base point $v_0$. 
(Since $R_0$ is the neutral element of $\Complex \R(K_n)$, we omit the identities for $R_0$.) 
The commutativity immediately follows from \eqref{eqSreg1} -- \eqref{eqSreg3}, and the associativity can be revealed by direct calculations. 
(Using Proposition \ref{PropAssCom} curtails necessary calculations.)


Similarly, if $X$ is distance-regular, we find that $\R (X)$ becomes a hermitian discrete hypergroup, whose structure is independent of $v_0$. 
Since 
\begin{align*}
\abs{\Gamma_{j} (v) \cap \Gamma_{k} (v_0)} =& \abs{\setcond{w \in V}{\dist{v}{w} = j, \dist{v}{v_0} = k}} \\
=& \abs{\setcond{w \in V}{(v, w) \in R_j, (w, v_0) \in R_k}} \\
=& p_{j,k}^{i}
\end{align*}
when $\dist{v}{v_0} = i$, we can compute $P_{i,j}^{k}$ to be 
\begin{equation} \label{eqCoefDRG}
P_{i,j}^{k} = 
\frac{1}{\abs{\Gamma_{i} (v_0)}} \sum_{v \in \Gamma_{i} (v_0)} \frac{\abs{\Gamma_{j} (v) \cap \Gamma_{k} (v_0)}}{\abs{\Gamma_{j} (v)}} 
= \frac{p_{j,k}^{i}}{p_{j,j}^{0}}. 
\end{equation}
(Pay attention to the subscripts in the numerator of the last member.) 
These constants are determined independently of a choice of $v_0$. 
Now, we show that the convolution $\circ$ is commutative and associative. 
The commutativity is deduced from the identities $\mathrm{(c)}$ and $\mathrm{(f)}$ in Proposition \ref{PropAssEq}; 
Since 
\begin{equation*}
p_{j,k}^{i} p_{i,i}^{0} = p_{j,i}^{k} p_{k,k}^{0} = p_{i,j}^{k} p_{k,k}^{0} = p_{i,k}^{j} p_{j,j}^{0}, 
\end{equation*}
we obtain that 
\begin{equation*}
P_{i,j}^{k} = \frac{p_{j,k}^{i}}{p_{j,j}^{0}} = \frac{p_{i,k}^{j}}{p_{i,i}^{0}} = P_{j,i}^{k}. 
\end{equation*}
This implies that $R_i \circ R_j = R_j \circ R_i$. 
For the associativity $(R_h \circ R_i) \circ R_j = R_h \circ (R_i \circ R_j)$, 
we use the identities $\mathrm{(c)}$, $\mathrm{(e)}$ and $\mathrm{(f)}$ in Proposition \ref{PropAssEq}. 
By direct calculations, we find that 
\begin{gather*}
(R_h \circ R_i) \circ R_j = \sum_{k \in I} \left( \sum_{l \in I} P_{h,i}^{l} P_{l,j}^{k} \right) R_k, \\
R_h \circ (R_i \circ R_j) = \sum_{k \in I} \left( \sum_{l \in I} P_{i,j}^{l} P_{h,l}^{k} \right) R_k, 
\end{gather*}
so it suffices to show that $\sum_{l \in I} P_{h,i}^{l} P_{l,j}^{k} = \sum_{l \in I} P_{i,j}^{l} P_{h,l}^{k}$ holds for all $k \in I$. 
Indeed, for any $k \in I$, computations proceed as follows:  
\begin{align*}
\sum_{l \in I} P_{h,i}^{l} P_{l,j}^{k} =& \frac{1}{p_{i,i}^{0} p_{j,j}^{0}} \sum_{l \in I} p_{j,k}^{l} p_{i,l}^{h} \\
=& \frac{1}{p_{i,i}^{0} p_{j,j}^{0}} \sum_{l \in I} p_{i,j}^{l} p_{l,k}^{h} \\
=& \frac{1}{p_{i,i}^{0} p_{j,j}^{0}} \sum_{l \in I} p_{j,i}^{l} p_{l,k}^{h} \\
=& \frac{1}{p_{i,i}^{0} p_{j,j}^{0}} \sum_{l \in I} \frac{p_{j,l}^{i} p_{i,i}^{0}}{p_{l,l}^{0}} \cdot p_{l,k}^{h} \\
=& \sum_{l \in I} \frac{p_{j,l}^{i}}{p_{j,j}^{0}} \cdot \frac{p_{l,k}^{h}}{p_{l,l}^{0}} \\
=& \sum_{l \in I} P_{i,j}^{l} P_{h,l}^{k}. 
\end{align*}
The above arguments provides the following proposition. 

\begin{theorem} \label{ThmDRGHG}
Let $X = (V, E)$ be a distance-regular graph and $v_0 \in V$ the base point. 
Then, $\R (X)$ becomes a hermitian discrete hypergroup. 
Moreover, the hypergroup structure of $\R (X)$ is independent of $v_0$. 
\end{theorem}

There is another way to construct a hermitian discrete hypergroup from a distance-regular graph, 
which was also introduced by Wildberger \cite{Wil95}. 
This method is based on the ``Bose-Mesner algebra,'' which is associated to an association scheme. 
Let $(Y, \R)$ be an association scheme, where $\R = \set{R_i}_{i \in I}$ with an index set $I$. 
We define matrices $A^{(i)} = \left( A_{x,y}^{(i)} \right)_{x, y \in Y} \in \Matsp{Y}{\Integer}$ for each $i \in I$ as 
\begin{equation*}
A_{x,y}^{(i)} = 
 \begin{cases}
  1 & \left( (x, y) \in R_i \right), \\
  0 & \left( (x, y) \notin R_i \right), 
 \end{cases}
\end{equation*}
like \eqref{eqDistmat}. 
Then, a product $A^{(i)} A^{(j)}$ of two matrices can be written in a linear combination of $A^{(k)}$'s as
\begin{equation*}
A^{(i)} A^{(j)} = \sum_{k \in I} p_{i,j}^{k} A^{(k)}, 
\end{equation*}
where $p_{i,j}^{k} = 0$ except for finitely many $k \in I$. 
Moreover, it can be verified that, for any $i$, $j \in I$, $A^{(i)} A^{(j)} = A^{(j)} A^{(i)}$ 
from the identity $\mathrm{(c)}$ in Proposition \ref{PropAssEq}. 
Hence we have that the $\Complex$-vector space $\bigoplus_{k \in I} \Complex A^{(k)}$ 
turns into an associative and commutative $\Complex$-algebra with respect to the ordinary addition and multiplication of matrices. 
Setting $C^{(i)} = (p_{i,i}^{0})^{-1} A^{(i)}$ for each $i \in I$, we have 
\begin{align}
C^{(i)} C^{(j)} =& \sum_{k \in I} \frac{p_{i,j}^{k} p_{k,k}^{0}}{p_{i,i}^{0} p_{j,j}^{0}} C^{(k)} \notag \\
=& \sum_{k \in I} \frac{p_{j,k}^{i}}{p_{j,j}^{0}} C^{(k)}  \label{eqCoefAss}
\end{align}
from $\mathrm{(c)}$ and $\mathrm{(f)}$ in Proposition \ref{PropAssEq}. 
When $(Y, \R) = (V, \R (X))$ for some distance-regular graph $X = (V, E)$, 
comparing \eqref{eqCoefDRG} and \eqref{eqCoefAss} yields 
that $\set{C^{(i)}}_{i \in I}$ has the same structure of a hermitian discrete hypergroup as that of $\R (X)$ 
. 


\subsection{Examples of distance-regular graphs} \label{SecDRGEx}

Let us now see examples of distance-regular graphs. 
It seems that graph theorists usually consider only finite ones so that there are known many examples of finite distance-regular graphs. 
Some elementary examples and infinite ones will be introduced in this section. 
For more examples, see e.g.\ \cite{Bro-Coh-Neu}. 

\begin{itemize}
\item[$\mathrm{(i)}$] Complete graphs
\end{itemize}

One of the simplest examples of distance-regular graphs are complete graphs. 
For $n \in \Natural$ with $n \geq 2$, let $K_n$ denote the complete graph with $n$ vertices. 
Then, $\diam{K_n} = 1$ and the intersection array of $K_n$ is $(n-1; 1)$. 
The canonical partition $V(K_n) \times V(K_n)$, where $V(K_n)$ denotes the vertex set of $K_n$, consists of two sets, 
the diagonal set $R_0$ and its complement $R_1$. 
The intersection numbers $p_{i,j}^{k}$ of $(V(K_n), \R(K_n))$ are given by 
\begin{gather*}
p_{0,0}^{0} = 1, \quad p_{0,0}^{1} = 0, \\
p_{0,1}^{0} = p_{1,0}^{0} = 0, \quad p_{0,1}^{1} = p_{1,0}^{1} = 1, \\
p_{1,1}^{0} = n - 1, \quad p_{1,1}^{1} = n - 2, 
\end{gather*}
so that we can compute the coefficients $P_{i,j}^{k} = p_{j,k}^{i} / p_{j,j}^{0}$ to be 
\begin{gather*}
P_{0,0}^{0} = 1, \quad P_{0,0}^{1} = 0, \\
P_{0,1}^{0} = P_{1,0}^{0} = 0, \quad P_{0,1}^{1} = P_{1,0}^{1} = 1, \\
P_{1,1}^{0} = \frac{1}{n-1}, \quad P_{1,1}^{1} = \frac{n-2}{n-1}. 
\end{gather*}
By these computations, the structure identitiy of $\R(K_n)$ turns out that 
\begin{equation*}
R_1 \circ R_1 = \frac{1}{n-1} R_0 + \frac{n-2}{n-1} R_1. 
\end{equation*}

This hypergroup $\R(K_n)$ is isomorphic to a special case of $\Integer_q (2)$, which is called the ``$q$-deformation of $\Integer/2\Integer$.'' 
For details of $q$-deformation hypergroups, see \cite{KTY15}, \cite{Tsu15} for example. 

\begin{itemize}
\item[$\mathrm{(ii)}$] Platonic solids
\end{itemize}

Consider the platonic solid with $n$ vertices (of course we only consider the cases when $n = 4$, $6$, $8$, $12$, $20$). 
Let $V_n$ be the set of its vertices and $E_n$ the set of its edges. 
The graph $\mathcal{S}_n = (V_n, E_n)$ is known to be distance-regular. 
As for the structure identities of $\R(\mathcal{S}_n)$, we refer to Wildberger's article \cite{Wil94}. 

\begin{itemize}
\item[$\mathrm{(iii)}$] Infinite regular trees
\end{itemize}

\begin{figure}[t]
\begin{minipage}{0.5\hsize}
\centering
\includegraphics{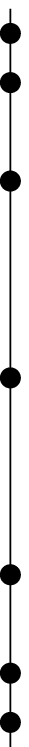}
\caption{Infinite $2$-regular tree}
\label{Fig2Tree}
\end{minipage}
\begin{minipage}{0.5\hsize}
\centering
\includegraphics{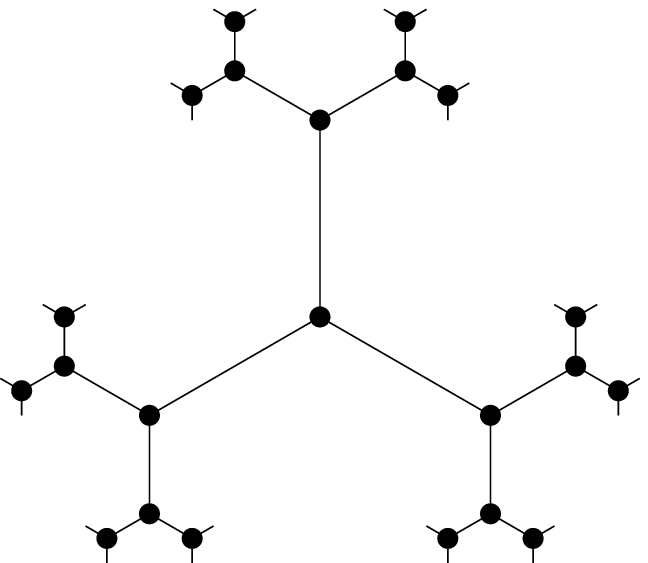}
\caption{Infinite $3$-regular tree}
\label{Fig3Tree}
\end{minipage}
\end{figure}

\begin{figure}
\centering
\includegraphics{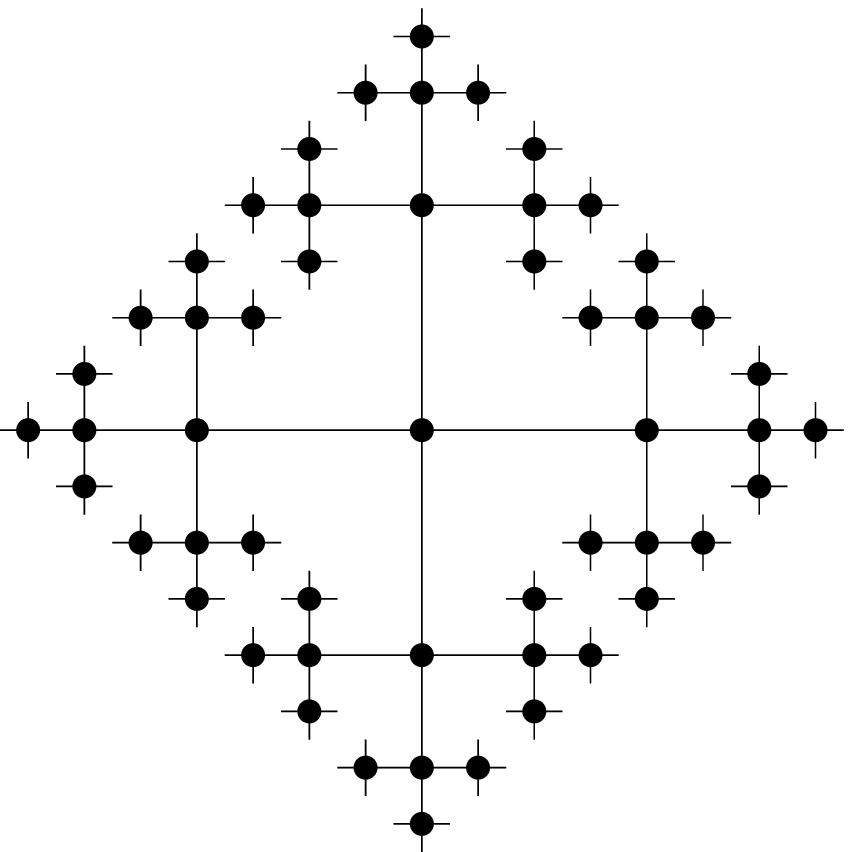}
\caption{Infinite $4$-regular tree}
\label{Fig4Tree}
\end{figure}

For each $n \in \Natural$ with $n \geq 2$, 
there exists a unique, but except for isomorphic ones, infinite $n$-regular connected graph without any cycles. 
(A \textit{cycle} means a path from a vertex to itself 
which does not pass through the same vertex twice except for the initial and the terminal vertex.)
We call such a graph the \textit{infinite $n$-regular tree}. 
(The term ``tree'' means a graph without any cycles.) 
For example, the infinite $2$-regular tree, $3$-regular tree and $4$-regular tree are partially drawn 
as in Figures \ref{Fig2Tree}, \ref{Fig3Tree} and \ref{Fig4Tree}, respectively.  
Let $\T_n = (V_n, E_n)$ denote the infinite $n$-regular tree and take an arbitrary vertex $v_0 \in V_n$ as the base point. 

We remark that $\T_n$ can be realized as a ``Cayley graph'' when $n$ is even. 
Cayley graphs arise here and in the following sections, so we recall the definition. 

\begin{definition} \label{DefCayley}
Let $G$ be a group and $\Omega$ be a subset of $G \setminus \set{1_{G}}$ satisfying 
\begin{equation} \label{eqInvClosed}
g \in G \Rightarrow g^{-1} \in G. 
\end{equation}
We define a graph $X = (V, E)$ as follows: 
\begin{itemize}
\item $V = G$. 
\item $E = \setcond{\set{g, h}}{g, h \in G,\, g^{-1} h \in \Omega}$. 
\end{itemize}
This graph $X$ is called the \textit{Cayley graph} and denoted by $\Cayley{G}{\Omega}$. 
\end{definition} 

A Cayley graph $\Cayley{G}{\Omega}$ must be $\abs{\Omega}$-regular and vertex-transitive. 
It becomes connected if and only if $\Omega$ generates $G$. 
 
For $m \in \Natural$, we let $F_m$ be the free group generated by $m$ symbols $g_1$, $g_2$, $\cdots$, $g_m$ 
and $\Omega_m = \tset{g_1^{\pm 1}, g_2^{\pm 1}, \cdots, g_m^{\pm 1}} \subset F_m$. 
The Cayley graph $\Cayley{F_m}{\Omega_m}$ is classically known to be isomorphic to the infinite $2m$-regular tree $\T_{2m}$. 

Now, we show that $\T_n$ is a distance-regular graph to prove the claim $\mathrm{(a)}$ of Theorem \ref{ThmInfty}. 
It will be proved in Proposition \ref{PropDRTree}. 
The following lemma gives some significant properties of the infinite regular trees and is useful before we prove that $\T_n$ is a distance-regular graph. 

\begin{lemma} \label{LemGeoUniqueTree}
Let $X = (V, E)$ be an infinite connected graph whose vertices all have finite degrees. 
\begin{itemize}
\item[$\mathrm{(a)}$] For every $v \in V$, we have $e(v) = \infty$. 
\item[$\mathrm{(b)}$] Suppose that $X$ has no cycles. 
Then, for any $v$, $w \in V$, there exists a unique geodesic from $v$ to $w$. 
\end{itemize}
\end{lemma}

\begin{proof}
The claim $\mathrm{(a)}$ follows from the proof of Proposition \ref{PropWellDefined}. 

We now give a proof of the claim $\mathrm{(b)}$. 
This is clear when $v = w$, so that we may assume that $\dist{v}{w} > 0$. 
Suppose that there would exist two distinct geodesics 
\begin{gather*}
v = v_0 \rightarrow v_1 \rightarrow \cdots \rightarrow v_{i-1} \rightarrow v_i = w, \\
v = w_0 \rightarrow w_1 \rightarrow \cdots \rightarrow w_{i-1} \rightarrow w_i = w
\end{gather*}
from $v$ to $w$, where $i = \dist{v}{w}$. 
(There exist at least one geodesics from $v$ to $w$ since $X$ is connected.) 
Then, we could find at least one $j \in \set{1, 2, \cdots, i-1}$ such that $v_j \neq w_j$. 
Letting $k^{*} = \max \setcond{k \in \set{1, 2, \cdots, j-1}}{v_k = w_k}$ for this $j$, we would obtain a cycle 
\begin{multline*}
v_{k^{*}} \to v_{k^{*} + 1} \to \cdots \to v_{j-1} \to v_{j} = w_{j} \\
\to w_{j-1} \to \cdots \to w_{k^{*} + 1} \to w_{k^{*}} = v_{k^{*}}, 
\end{multline*}
which contradicts the assumption that $X$ has no cycles. 
\end{proof}

By using this lemma, we can prove that $\T_n$ is distance-regular. 

\begin{proposition} \label{PropDRTree}
Let $n \in \Natural$ with $n \geq 2$. 
Then the infinite $n$-regular tree $\T_n$ is a distance-regular graph with the intersection array $(n, n-1, n-1, n-1, \cdots ; 1, 1, 1, 1, \cdots)$ 
(i.e.\ $b_0 = n$, $b_1 = b_2 = b_3 = \cdots = n - 1$, $c_1 = c_2 = c_3 = \cdots = 1$). 
\end{proposition}

\begin{proof}
We calculate the entries $b_i$'s and $c_i$'s of the intersection array of $\T_n$. 
Since $\T_n$ is $n$-regular, we have $b_0 = n$. 

Now we take arbitrary two vertices $v$, $w \in V_n$ of distance $i \geq 1$. 
We are going to show that 
$\abs{\Gamma_{i+1} (v) \cap \Gamma_{1} (w)} = n - 1$ and $\abs{\Gamma_{i-1} (v) \cap \Gamma_{1} (w)} = 1$. 

By Lemma \ref{LemGeoUniqueTree} $\mathrm{(b)}$, there exists a unique geodesic 
$v = v_0 \rightarrow v_1 \rightarrow \cdots \rightarrow v_{i-1} \rightarrow v_i = w$ from $v$ to $w$. 
Since $v_{i-1} \in \Gamma_{i-1} (v) \cap \Gamma_{1} (w)$, we have $\abs{\Gamma_{i-1} (v) \cap \Gamma_{1} (w)} \geq 1$. 
If $\Gamma_{i-1} (v) \cap \Gamma_{1} (w)$ would possess another vertex $w_{i-1}$, there would exist the second geodesic 
$v = w_0 \rightarrow w_1 \rightarrow \cdots \rightarrow w_{i-1} \rightarrow w_i = w$ from $v$ to $w$. 
This contradicts the uniqueness of the geodesic from $v$ to $w$, so we find that $\abs{\Gamma_{i-1} (v) \cap \Gamma_{1} (w)} = 1$. 

Next, we check that $\Gamma_{i} (v) \cap \Gamma_{1} (w) = \varnothing$. 
If there would exist a vertex $u \in \Gamma_{i} (v) \cap \Gamma_{1} (w)$, 
we would find a geodesic $v = u_0 \rightarrow u_1 \rightarrow \cdots \rightarrow u_{i-1} \rightarrow u_i = u$ from $v$ to $u$ 
and a geodesic $v = w_0 \rightarrow w_1 \rightarrow \cdots \rightarrow w_{i-1} \rightarrow w_i = w$ from $v$ to $w$. 
Letting $k^{*} = \max \setcond{k \in \set{1, 2, \cdots, i-1}}{u_k = w_k}$, we would obtain a cycle 
\begin{multline*}
u_{k^{*}} \rightarrow u_{k^{*} + 1} \rightarrow \cdots \rightarrow u_{i-1} \rightarrow u_i = u \\ 
\rightarrow w = w_i \rightarrow w_{i-1} \rightarrow \cdots \rightarrow w_{k^{*} + 1} \rightarrow w_k = u_k, 
\end{multline*}
which contradicts that $\T_n$ has no cycles. 
Therefore, it turns out that $\Gamma_{i} (v) \cap \Gamma_{1} (w) = \varnothing$. 

We find that $\Gamma_{1} (w)$ is partitioned into mutually disjoint three subsets, that is, 
\begin{equation*}
\Gamma_{1} (w) = \left( \Gamma_{i-1} (v) \cap \Gamma_{1} (w) \right) \sqcup \left( \Gamma_{i} (v) \cap \Gamma_{1} (w) \right) 
\sqcup \left( \Gamma_{i+1} (v) \cap \Gamma_{1} (w) \right). 
\end{equation*}
Hence we have 
\begin{align*}
\abs{\Gamma_{i+1} (v) \cap \Gamma_{1} (w)} &= 
\abs{\Gamma_{1} (w)} - \abs{\Gamma_{i-1} (v) \cap \Gamma_{1} (w)} - \abs{\Gamma_{i} (v) \cap \Gamma_{1} (w)} \\
&= n - 1. 
\end{align*}
These computations work for any $v$ and $w$ with $\dist{v}{w} = i$, so that we obtain that $b_i = n - 1$ and $c_i = 1$. 
\end{proof}

To obtain the structure identities of $\R (\T_n)$, we have to compute the intersection numbers $p_{i,j}^{k}$. 

\begin{proposition} \label{PropCoefAssTree}
Let $n \in \Natural$ with $n \geq 2$ and $p_{i,j}^{k}$ be the intersection numbers of the association scheme $(V_n, \R(\T_n))$. 
\begin{itemize}
\item[$\mathrm{(a)}$] For any $i \in \Natural$, we have $p_{i,i}^{0} = n (n - 1)^{i-1}$. 
\item[$\mathrm{(b)}$] For any $i$, $j$, $k \in \Natural$, we can get the following: 
\begin{equation} \label{eqCoefAssTree1}
p_{i,j}^{k} = 
  \begin{cases}
   (n - 1)^{i} & (j = i + k), \\
   (n - 2) (n - 1)^{i-h-1} & (j = i + k - 2h,\, 0 < h < \min \set{i, k}), \\
   (n - 1)^{i - \min \set{i, k}} & (j = \abs{i - k}), \\
   0 & (\text{otherwise}). 
  \end{cases}
\end{equation}
\end{itemize}
\end{proposition}

\begin{proof}
Take an arbitrary vertex $v_0 \in V_n$ as a base point. 
First, we compute $p_{i,i}^{0}$. 

We have $p_{1,1}^{0} = \abs{\Gamma_{1} (v_0)} = \deg (v_0) = n$, so the desired counting is obtained when $i = j = 1$. 
Now, we show that $p_{i,i}^{0} = n (n-1)^{i-1}$ by induction on $i \in \Natural$. 
We assume that $p_{i-1,i-1}^{0} = n (n - 1)^{i-2}$ for some $i \geq 2$. 
Take $w \in \Gamma_{i-1}(v_0)$ arbitrarily. 
Then, by Proposition \ref{PropDRTree}, 
just one vertex belongs to $\Gamma_{i-2}(v_0)$ among $n$ vetices adjacent to $w$, and the other $n-1$ vertices belong to $\Gamma_{i}(v_0)$. 
Thus we get $\abs{\Gamma_{i}(v_0) \cap \Gamma_{1}(w)} = n - 1$. 
If we take another vertex $w' \in \Gamma_{i-1}(v_0)$, 
it turns out that $\Gamma_{i}(v_0) \cap \Gamma_{1}(w)$ and $\Gamma_{i}(v_0) \cap \Gamma_{1}(w')$ must be disjoint. 
Hence we obtain 
\begin{align*}
\abs{\Gamma_{i}(v_0)} &= \abs{\Gamma_{i}(v_0) \cap \bigcup_{w \in \Gamma_{i-1}(v_0)} \Gamma_{1}(w)} \\
&= \abs{\bigcup_{w \in \Gamma_{i-1}(v_0)} \left( \Gamma_{i}(v_0) \cap \Gamma_{1}(w) \right)} \\
&= \sum_{w \in \Gamma_{i-1}(v_0)} \abs{\Gamma_{i}(v_0) \cap \Gamma_{1}(w)} \\
&= \abs{\Gamma_{i}(v_0) \cap \Gamma_{1}(w)} \cdot p_{i-1,i-1}^{0}
&= n (n-1)^{i-1} 
\end{align*}
from the induction hypothesis. 
 
Secondly, we show \eqref{eqCoefAssTree1}. 
We now determine the cases when $p_{i,j}^{k} = \abs{\Gamma_{i} (v) \cap \Gamma_{j} (v_0)} = 0$ for $v \in \Gamma_{k} (v_0)$. 
We note that there exists a one-to-one correspondence between $V_n$ and the set of all geodesics on $\T_n$ starting at $v_0$, 
which is derived from Lemma \ref{LemGeoUniqueTree} $\mathrm{(b)}$. 
Suppose that $w \in \Gamma_{i} (v) \cap \Gamma_{j} (v_0)$. 
Then, we can take the unique geodesic 
\begin{equation*}
v = u_0 \to u_1 \to \cdots \to u_{k-1} \to u_k = v_0
\end{equation*}
from $v$ to $v_0$ and also the unique geodesic 
\begin{equation*}
v = w_0 \to w_1 \to \cdots \to w_{i-1} \to w_i = w
\end{equation*}
from $v$ to $w$. 

If $u_0 = w_0$, $u_1 = w_1$, $\cdots$, $u_h = w_h$ and $u_{h+1} \neq w_{h+1}$ hold for some $0 \leq h < \min \set{i, k}$, then the path 
\begin{multline*}
v_0 = u_k \to u_{k-1} \to \cdots \to u_{h+1} \to u_h = w_h \\
\to w_{h+1} \to \cdots \to w_{i-1} \to w_i = w
\end{multline*}
must be the geodesic from $v_0$ to $w$. 
In these cases, we have $j = \dist{v_0}{w} = i + k - 2h$. 
(Remark that we always have $u_0 = w_0 = v$.)

Similarly, if $u_0 = w_0$, $u_1 = w_1$, $\cdots$, $u_{\min \set{i, k}} = w_{\min \set{i, k}}$, then the path 
\begin{multline*}
v_0 = u_k \to u_{k-1} \to \cdots \to u_{{\min \set{i, k}}+1} \to u_{\min \set{i, k}} = w_{\min \set{i, k}} \\ 
\to w_{{\min \set{i, k}}+1} \to \cdots \to w_{i-1} \to w_i = w
\end{multline*}
must be the geodesic from $v_0$ to $w$, and we get $j = \dist{v_0}{w} = i + k - 2 \min \set{i, k} = \abs{i - k}$. 
Therefore, we find that $w \in \Gamma_{i} (v) \cap \Gamma_{j} (v_0)$ for $v \in \Gamma_{k} (v_0)$ 
only if $j = i + k -2h$ for some $h \in \Nonnegative$ such that $0 \leq h \leq \min \set{i, k}$. 
In other words, when $v \in \Gamma_{k} (v_0)$, we obtain that $p_{i,j}^{k} = \abs{\Gamma_{i} (v) \cap \Gamma_{j} (v_0)} = 0$ 
unless $j = i + k -2h$ for some $h \in \Nonnegative$ such that $0 \leq h \leq \min \set{i, k}$. 

Next, we calculate $\abs{\Gamma_{i} (v) \cap \Gamma_{i+k-2h} (v_0)}$ for $h = 0$, $1$, $\cdots$, $\min \set{i, k}$. 
Let the path
\begin{equation*}
v = u_0 \to u_1 \to \cdots \to u_{i-1} \to u_k = v_0
\end{equation*}
be the unique geodesic from $v$ to $v_0$. 
Then, by the above argument, we find that $\abs{\Gamma_{i} (v) \cap \Gamma_{i+k-2h} (v_0)}$ is equal to the number of geodesics 
\begin{equation} \label{eqGeod}
v = w_0 \to w_1 \to \cdots \to w_{i-1} \to w_i = w
\end{equation}
with ($u_0 = w_0$,) $u_1 = w_1$, $\cdots$, $u_h = w_h$ and $u_{h+1} \neq w_{h+1}$ when $h < \min \set{i, k}$. 

When $0 < h < \min \set{i, k}$, each of such geodesics satisfies neither $w_{h+1} = u_{h+1}$ nor $w_{h+1} = w_{h-1}$. 
(It requires the latter condition that the path \eqref{eqGeod} is a geodesic. 
Since $\dist{v}{u_{h+1}} = h + 1 \neq h - 1 = \dist{v}{w_{h-1}}$, we remark that $u_{h+1} \neq w_{h-1}$.) 
Thus the number of possible vertices as $w_{h+1}$ is equal to $n - 2$. 
Furthermore, each of $w_{h+2}$, $\cdots$, $w_i (= w)$ can be chosen from exactly $n - 1$ candidates. 
(The candidates of $w_{h+2}$, $\cdots$, $w_i$ exclude a vertex which coincides with $w_h$, $\cdots$, $w_{i-2}$, respectively.) 
These observations yield that 
\begin{align*}
&\abs{\Gamma_{i} (v) \cap \Gamma_{i+k-2h} (v_0)} \\
=& \abs{\setcond{(w_{h+1}, \cdots, w_{i}) \in V_{n}^{i-h}}
                        {w_{h+1} \neq u_{h+1}, w_{h+1} \neq w_{h-1}, w_{h+2} \neq w_{h}, \cdots, w_{i} \neq w_{i-2}}} \\
=& (n-2) (n-1)^{i-l-1}
\end{align*}
for $0 < h < \min \set{i, k}$. 

When $h = 0$, the number of candidates of $w_1$ changes into $n - 1$ 
since we can arbitrarily take $w_1 \in \Gamma_{1} (w_0)$ except for $u_1$. 
Thus we obtain that 
\begin{align*}
&\abs{\Gamma_{i} (v) \cap \Gamma_{i+k} (v_0)} \\
=& \abs{\setcond{(w_{1}, \cdots, w_{i}) \in V_{n}^{i}}{w_{1} \neq u_{1}, w_{2} \neq w_{0}, \cdots, w_{i} \neq w_{i-2}}} \\
=& (n-1)^{i}. 
\end{align*}

When $i \leq k$ and $h = \min \set{i, k} = i$, all of $w_1$, $w_2$, $\cdots$, $w_i$ are uniquely determined, and we have that 
\begin{equation*}
\abs{\Gamma_{i} (v) \cap \Gamma_{\abs{i-k}} (v_0)} = 1 = (n-1)^{i-{\min \set{i, k}}}. 
\end{equation*}
When $i > k$ and $h = \min \set{i, k} = k$, the first $k$ vertices $w_1$, $w_2$, $\cdots$, $w_k$ are uniquely determined 
and each of $w_{k+1}$, $\cdots$, $w_i$ can be chosen from exactly $n - 1$ candidates. 
(The candidates of $w_{k+1}$, $\cdots$, $w_i$ exclude a vertex which coincides with $w_{k-1}$, $\cdots$, $w_{i-2}$, respectively.) 
Thus we obtain that 
\begin{align*}
&\abs{\Gamma_{i} (v) \cap \Gamma_{\abs{i-k}} (v_0)} \\
=& \abs{\setcond{(w_{k+1}, \cdots, w_{i}) \in V_{n}^{i-k}}{w_{k+1} \neq w_{k-1}, w_{k+2} \neq w_{k}, \cdots, w_{i} \neq w_{i-2}}} \\
=& (n-1)^{i-k} \\
=& (n-1)^{i-{\min \set{i, k}}}. 
\end{align*}
\end{proof}

The coefficients $P_{i,j}^{k}$ of the structure identities of $\R(\T_n)$ can be easily computed from Proposition \ref{PropCoefAssTree}, 
so that we obtain the following.  

\begin{corollary} \label{CorCoefTree}
Let $n \in \Natural$ with $n \geq 2$, $i$, $j \in \Natural$ and $k \in \Nonnegative$. 
Then, the coefficients $P_{i,j}^{k}$ of structure identities of $\R(\T_n)$ are given by 
\begin{equation*}
P_{i,j}^{k} = 
  \begin{cases}
   (n-1) / n & (k = i + j), \\
   (n-2) / n (n-1)^{h} & (k = i + j -2h,\, 0 < h < \min \set{i, j}), \\
   1 / n (n-1)^{\min \set{i, j} - 1} & (k = \abs{i-j}), \\
   0 & (\text{otherwise}).
 \end{cases}
\end{equation*}
\end{corollary}

\begin{itemize}
\item[$\mathrm{(iv)}$] Linked triangles 
\end{itemize}

\begin{figure}[t]
\begin{minipage}{0.5\hsize}
\centering
\includegraphics{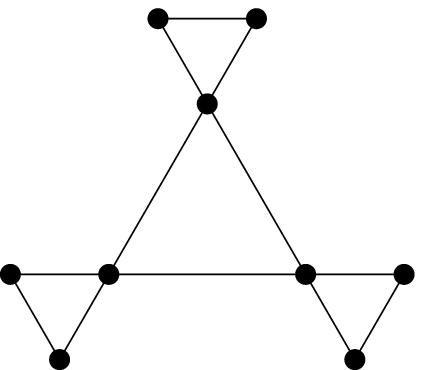}
\caption{Construction of $X_1$}
\label{FigLink1}
\end{minipage}
\begin{minipage}{0.5\hsize}
\centering
\includegraphics{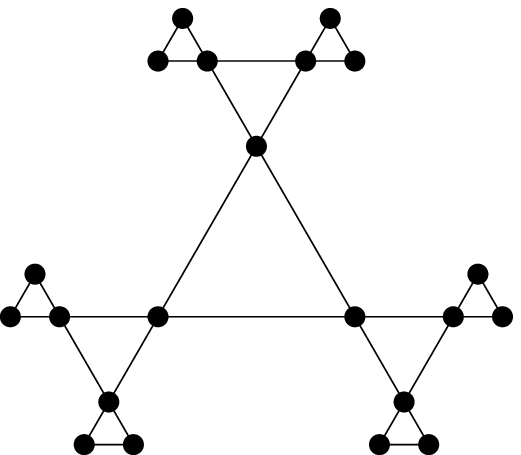}
\caption{Construction of $X_2$}
\label{FigLink2}
\end{minipage}
\end{figure}


Here, we see an example of infinite distance-regular graphs with cycles; 
Such an example can be constructed by linking triangular graphs inductively. 

As the beginning step, we let $X_0 = K_3 = (V_0, E_0)$. 
We make three copies of $K_3$ ``linked'' to $X_0$ to share distinct vertices. 
Then, we have a graph with nine vertices and nine edges. 
(See Figure \ref{FigLink1}.)
Let it be denoted by $X_1$. 
Next, we make copies of $K_3$ linked to each vertex of $X_1 = (V_1, E_1)$ that is not yet linked with another triangle. 
We need six copies of $K_3$ in this step to obtain the graph $X_2 = (V_2, E_2)$ with twenty-one vertices and thirty edges. 
(See Figure \ref{FigLink2}.)
Repeating this process, we have an ascending sequence $\set{X_n = (V_n, E_n)}_{n \in \Nonnegative}$ of finite graphs. 
We let an infinite graph $X_{\infty} = (V_{\infty}, E_{\infty})$ be defined by the union of all of $X_n$: 
It consists of the vertex set $V_{\infty} = \bigcup_{n \in \Nonnegative} V_n$ and the edge set $E_{\infty} = \bigcup_{n \in \Nonnegative} E_n$. 
In this article, we call the graph $X_{\infty}$ the ``linked-triangle graph.''

We shall show that $X_{\infty}$ is a distance-regular graph. 
Preparatory to the proof, we label each element of $V_{\infty}$ with a three-letter word. 
Consider the set $W$ of non-empty words of finite length composed of the three letters $a$, $b$ and $c$, 
in which every two consequent letters differ. 
We regard $W$ as a vertex set, and two words $v = l_1 l_2 \cdots l_m$, $w = l'_1 l'_2 \cdots l'_n \in W$, 
where $l_1, l_2, \cdots, l_m, l'_1, l'_2, \cdots, l'_n \in \set{a, b, c}$, 
are said to be adjacent if and only if one of the followings is satisfied: 
\begin{itemize}
\item It holds that $\abs{m-n} = 1$ and that $l_1 = l'_1$, $l_2 = l'_2$, $\cdots$, $l_{\min \set{m, n}} = l'_{\min \set{m, n}}$. 
\item It holds that $m = n$ and that $l_1 = l'_1$, $l_2 = l'_2$, $\cdots$, $l_{m-1} = l'_{m-1}$. 
\end{itemize}
Then, we have an infinite graph isomorphic to $X_{\infty}$, so that we regard every vertex of $X_{\infty}$ as labeled by an element of $W$. 
We identify the vertex set $V_{\infty}$ and the word set $W$ in the following arguments. 
In what follows, we use the notation $\length{v}$ for the length of the word $v \in W = V_{\infty}$. 

To show that $X_{\infty}$ is distance-regular, we have to grasp the basic properties of $X_{\infty}$ in the following lemma. 

\begin{lemma} \label{LemGeoUniqueLink}
\begin{itemize}
\item[$\mathrm{(a)}$] The linked-triangle graph $X_{\infty}$ admits no cycles of length four or greater. 
\item[$\mathrm{(b)}$] The linked-triangle graph $X_{\infty}$ is connected, 
and the geodesic from $v$ to $w$ in $X_{\infty}$ is unique for any $v$, $w \in V_{\infty}$. 
\end{itemize}
\end{lemma}

\begin{proof}
To prove the claim $\mathrm{(a)}$, we suppose that $X_{\infty}$ would contain a cycle 
\begin{equation*}
v_0 \to v_1 \to v_2 \to \cdots \to v_{L-1} \to v_0 
\end{equation*}
of length $L \geq 4$ with $v_0 = l_1 l_2 \cdots l_n$, where $n \in \Natural$ and $l_1$, $l_2$, $\cdots$, $l_n \in \set{a, b, c}$. 
Changing the initial vertex if necessary, we may assume that $\length{v_0} \geq \length{v_i}$ for $i = 1$, $2$, $\cdots$, $L-1$. 
If $n = 1$, then all of $v_i$'s are distinct and belong to $\set{a, b, c}$, but this is impossible. 
If $n \geq 2$, then, since both $v_1$ and $v_{L-1}$ are adjacent to $v_0$, 
we find that one of them should be $l_1 l_2 \cdots l_{n-1}$ and that the other $l_1 l_2 \cdots l'_n$ for $l'_n \neq l_n, l_{n-1}$. 
Changing the direction of the cycle if necessary, we may assume that $v_1 = l_1 l_2 \cdots l_{n-1}$ and $v_{L-1} = l_1 l_2 \cdots l'_n$. 
Since $v_{L-1}$ is adjacent to $v_0$, $v_1$ and $v_{L-2}$, and $v_{L-2} \neq v_0, v_1$, 
we would obtain that $\length{v_{L-2}} = \length{v_{L-1}} + 1 > \length{v_0}$. 
This is a contradiction to the maximality of $\length{v_0}$, so that we obtain the claim $\mathrm{(a)}$. 

To show the claim $\mathrm{(b)}$, 
we let the word representations of $v$ and $w$ be denoted by $v = l_1 l_2 \cdots l_m$ and $w = l'_1 l'_2 \cdots l'_n$, respectively, 
where $m$, $n \in \Natural$, and $l_1, l_2, \cdots, l_m, l'_1, l'_2, \cdots, l'_n \in \set{a, b, c}$. 
In addition, set 
\begin{equation*}
k^{*} = 
 \begin{cases}
  \min \setcond{1 \leq k \leq \min \set{m, n}}{l_k \neq l'_k} & (l_k \neq l'_k \ \text{for some}\ 1 \leq k \leq \min \set{m, n}) \\
  \min \set{m, n} + 1 & (l_k = l'_k \ \text{for}\ 1 \leq k \leq \min \set{m, n}). 
 \end{cases}
\end{equation*}
When $k^{*} = 1$, one can take a path  
\begin{multline*}
v = l_1 l_2 \cdots l_m \rightarrow l_1 l_2 \cdots l_{m-1} \rightarrow \cdots \rightarrow l_1 l_2 \rightarrow l_1 \\
\rightarrow l'_1 \rightarrow l'_1 l'_2 \rightarrow \cdots \rightarrow l'_1 l'_2 \cdots l'_{n-1} \rightarrow l'_1 l'_2 \cdots l'_n = w
\end{multline*}
from $v$ to $w$, and when $k^{*} \geq 2$, one can take a path 
\begin{multline*}
v = l_1 l_2 \cdots l_m \rightarrow l_1 l_2 \cdots l_{m-1} \rightarrow \cdots 
\rightarrow l_1 l_2 \cdots l_{k^{*}} \rightarrow l_1 l_2 \cdots l_{k^{*}-1} \\
= l'_1 l'_2 \cdots l'_{k^{*}-1} \rightarrow l'_1 l'_2 \cdots l'_{k^{*}} \cdots \rightarrow l'_1 l'_2 \cdots l'_{n-1} \rightarrow l'_1 l'_2 \cdots l'_n = w
\end{multline*}
from $v$ to $w$ . 
Therefore, $X_{\infty}$ is connected, so that there exists at least one geodesic from $v$ to $w$. 

Two distinct geodesics from $v$ to $w$ would allow us to find a cycle of length four or greater, 
whose existence contradicts to the claim $\mathrm{(a)}$ of Lemma \ref{LemGeoUniqueLink}.
\end{proof}

We can prove that $X_{\infty}$ is distance-regular by an argument similar to that for Lemma \ref{PropDRTree}. 

\begin{proposition} \label{PropDRLink}
The linked-triangle graph $X_{\infty}$ is a distance-regular graph with the intersection array $(4, 2, 2, 2, \cdots ; 1, 1, 1, 1, \cdots)$ 
(i.e.\ $b_0 = 4$, $b_1 = b_2 = b_3 = \cdots = 2$, $c_1 = c_2 = c_3 = \cdots = 1$). 
\end{proposition}

\begin{proof}
Since every vertex in $V_{\infty}$ possesses exactly four neighborhoods, we have $b_0 = 4$. 

We show that $c_i = 1$ for any $i \in \Natural$. 
Take an $i \in \Natural$ and two vertices $v$, $w \in V_{\infty}$ with $\dist{v}{w} = i$ arbitrarily. 
Then, there exists a unique geodesic 
\begin{equation*}
v = v_0 \to v_1 \to \cdots \to v_{i-1} \to v_i = w
\end{equation*}
from $v$ to $w$. 
Noting that $v_{i-1} \in \Gamma_{i-1}(v) \cap \Gamma_{1}(w)$, we find that $\abs{\Gamma_{i-1}(v) \cap \Gamma_{1}(w)} \geq 1$. 
On the other hand, the uniqueness of the geodesic ensures that $\abs{\Gamma_{i-1}(v) \cap \Gamma_{1}(w)} \leq 1$, 
so we have $\abs{\Gamma_{i-1}(v) \cap \Gamma_{1}(w)} = 1$. 
This means that $c_i = 1$. 

Next, we shall compute $\abs{\Gamma_{i+1}(v) \cap \Gamma_{1}(w)}$ for $i \in \Natural$ and $v$, $w \in V_{\infty}$ with $\dist{v}{w} = i$. 
Take the unique geodesic 
\begin{equation*}
v = v_0 \to v_1 \to \cdots \to v_{i-1} \to v_i = w
\end{equation*}
from $v$ to $w$. 
It follows that $v_{i-1} \notin \Gamma_{i+1}(v) \cap \Gamma_{1}(w)$ from the previous argument. 
Noting that $v_{i-1}$ and $w$ are adjacent and that there exists just one $3$-cycle containing $v_{i-1}$ and $w$, 
we find that there exists exactly one vertex $w'$ that is adjacent to both $v_{i-1}$ and $w$. 
Indeed, if there were two vertices $w'$ and $w''$ such that both $v_{i-1}$ and $w$ are adjacent to each of them, 
then we could find a $4$-cycle $v_{i-1} \to w' \to w \to w'' \to v_{i-1}$. 
This contradicts to the claim $\mathrm{(b)}$ of Lemma \ref{LemGeoUniqueLink}, hence we have the uniqueness of $w'$. 
The path $v \to v_1 \to \cdots \to v_{i-1} \to w'$ should be the geodesic from $v$ to $w'$, 
so that $w'$ belongs to $\Gamma_{i}(v) \cap \Gamma_{1}(w)$. 
The other two neighborhoods of $w$ belong to neither $\Gamma_{i-1}(v)$ nor $\Gamma_{i}(v)$, 
hence we find that they must belong to $\Gamma_{i+1}(v) \cap \Gamma_{1}(w)$ and $\abs{\Gamma_{i+1}(v) \cap \Gamma_{1}(w)} = 2$. 
This means that $b_i = 2$. 
\end{proof}

Let us determine the coefficients of the structure identities of $\R(X_{\infty})$. 
We can use the word representations to calculate the distance between two vertices of $X_{\infty}$. 

\begin{lemma} \label{LemDistLink}
Let $v$, $w \in V_{\infty}$ 
and $v = l_1 l_2 \cdots l_m$ and $w = l'_1 l'_2 \cdots l'_n$ be their word representations, 
where $m$, $n \in \Natural$ and $l_1$, $l_2$, $\cdots$, $l_m$, $l'_1$, $l'_2$, $\cdots$, $l'_n \in \set{a, b, c}$. 
Then, 
\begin{equation} \label{eqDistLink}
\dist{v}{w} = 
 \begin{cases}
  m + n - 2k^{*} - 1 & (l_i \neq l'_i \ \text{for some}\ 1 \leq i \leq \min \set{m, n}), \\
  \abs{m - n} & (l_i = l'_i \ \text{for}\ 1 \leq i \leq \min \set{m, n}), 
 \end{cases}
\end{equation}
where $k^{*} = \min \setcond{1 \leq k \leq \min \set{m, n}}{l_k \neq l'_k}$ 
in the case when $l_i \neq l'_i$ for some $1 \leq i \leq \min \set{m, n}$. 
\end{lemma}

\begin{proof}
It is obvious that \eqref{eqDistLink} holds when $v = w$, so we may assume that $v \neq w$. 
The unique geodesic from $v$ to $w$ is given by 
\begin{multline*}
v = l_1 l_2 \cdots l_m \rightarrow l_1 l_2 \cdots l_{m-1} \rightarrow \cdots 
\rightarrow l_1 l_2 \cdots l_{k^{*}} \rightarrow l_1 l_2 \cdots l_{k^{*}-1} \\
= l'_1 l'_2 \cdots l'_{k^{*}-1} \rightarrow l'_1 l'_2 \cdots l'_{k^{*}} \cdots \rightarrow l'_1 l'_2 \cdots l'_{n-1} \rightarrow l'_1 l'_2 \cdots l'_n = w
\end{multline*}
if $l_i \neq l'_i$ for some $1 \leq i \leq \min \set{m, n}$ and by
\begin{multline*}
v = l_1 l_2 \cdots l_m \rightarrow l_1 l_2 \cdots l_{m-1} \rightarrow \cdots \rightarrow l_1 l_2 \rightarrow l_1 \\
\rightarrow l'_1 \rightarrow l'_1 l'_2 \rightarrow \cdots \rightarrow l'_1 l'_2 \cdots l'_{n-1} \rightarrow l'_1 l'_2 \cdots l'_n = w
\end{multline*}
if $l_i = l'_i$ for $i = 1, \cdots, \min \set{m, n}$. 
The desired evaluations are obtained from these observations. 
\end{proof}

\begin{proposition} \label{PropCoefAssLink}
Let $p_{i,j}^{k}$ be the intersection numbers of the association scheme $(V_{\infty}, \R(X_{\infty}))$. 
\begin{itemize}
\item[$\mathrm{(a)}$] For any $i \in \Natural$, we have $p_{i,i}^{0} = 2^{i+1}$. 
\item[$\mathrm{(b)}$] For $i$, $j \in \Natural$ and $k \in \Nonnegative$, we can get the following: 
\begin{equation} \label{eqCoefAssLink1}
p_{i,j}^{k} = 
  \begin{cases}
   2^{\max \set{i-k, 0}} & (j = \abs{i - k}), \\
   2^{\max \set{i-k, 0}+h-1} & (j = \abs{i - k} + 2h - 1,\, 1 \leq h \leq \min \set{i, k}), \\
   2^{i} & (j = i + k), \\
   0 & (\text{otherwise}). 
  \end{cases}
\end{equation}
\end{itemize}
\end{proposition}

\begin{proof}
Take two vertices $v$, $w \in V_{\infty}$ with $\dist{v}{w} = k$. 
Since $X_{\infty}$ is a distance-regular graph, 
we may assume that the word representations of $v$ and $w$ are given by $v = a l_1 l_2 \cdots l_k$ and $w = a$, respectively, 
where $l_1$, $l_2$, $\cdots$, $l_k \in \set{a, b, c}$. 

We now classify the vertices belonging to $\Gamma_{i}(v)$ by distance from $w = a$. 
We implicitly use Lemma \ref{LemDistLink} in the following arguments. 

First, we consider the case when $i \leq k$. 
There are three ways to get a vertex in $\Gamma_{i}(v)$; 
The first one is to delete the last $i$ letters $l_k$, $\cdots$, $l_{k-i+1}$ from $v$. 
The second one is to add $i$ letters $l'_{k+1}$, $\cdots$, $l'_{k+i}$ to the tail of $v$. 
The third one is, for $1 \leq h \leq i$, to add $h - 1$ letters $l'_{k-i+h+1}$, $\cdots$, $l'_{k-i+2h-1}$ to the tail 
after deleting the last $i - h$ letters $l_{k}$, $\cdots$, $l_{k-i+h+1}$ from $v$ 
and changing the last letter $l_{k-i+h}$ of the remaining word into the other letter $l'_{k-i+h}$. 
The numbers of vertices provided by each way are found to be $1$, $2^{i}$ and $2^{h-1}$, 
and the distances between $a$ and a vertex provided by each way are $k - i$, $k + i$ and $k - i + 2h - 1$, respectively. 

Next, we consider the case when $i > k$. 
In this case, there are three ways to get a vertex in $\Gamma_{i}(v)$ given as follows; 
The first one is to add $i - k - 1$ letters $l'_{1}$, $\cdots$, $l'_{i-k-1}$ to the tail of $v$ 
after deleting the last $k$ letters $l_{k}$, $\cdots$, $l_{1}$ from $v$ 
and changing the remaining letter $a$ into the other letter $l'_{0}$. 
The second one is to add $i$ letters $l'_{k+1}$, $\cdots$, $l'_{k+i}$ to the tail of $v$. 
The third one is, for $1 \leq h \leq k$, to add $i - k + h - 1$ letters $l'_{h+1}$, $\cdots$, $l'_{i-k+2h-1}$ to the tail 
after deleting the last $k - h$ letters $l_{k}$, $\cdots$, $l_{h+1}$ from $v$ 
and changing the last letter $l_{h}$ of the remaining word into the other letter $l'_{h}$. 
The numbers of vertices provided by each way are found to be $2^{i-k}$, $2^{i}$ and $2^{i-k+h-1}$ 
(note that there are two candidates for $l'_{0}$ in the explanation of the first way), 
and the distances between $a$ and a vertex provided by each way are $i - k$, $i + k$ and $i - k + 2h - 1$, respectively. 
The conclusion \eqref{eqCoefAssLink1} can be deduced from these observations. 
\end{proof}

By Proposition \ref{PropCoefAssLink}, we can compute the structure identities of $\R(X_{\infty})$. 

\begin{corollary} \label{CorCoefLink}
Let $i$, $j \in \Natural$ and $k \in \Nonnegative$. 
Then, the coefficients $P_{i,j}^{k}$ of structure identities of $\R(X_{\infty})$ are given by 
\begin{equation*}
P_{i,j}^{k} = 
  \begin{cases}
   1/2 & (k = i + j), \\
   1/2^{\min \set{i, j} + 2 - h} & (k = \abs{i - j} + 2h - 1, 1 \leq h \leq \min \set{i, j}), \\
   1/2^{\min \set{i, j} + 1} & (k = \abs{i - j}), \\
   0 & (\text{otherwise}).
 \end{cases}
\end{equation*}
\end{corollary}


\section{Non-distance-regular graphs producing a hypergroup} \label{SecNDR}

In the previous section, we certified that a random walk on any distance-regular graph produces a hermitian discrete hypergroup. 
However, we should note that a random walk on a certain non-distance-regular graph produces a hermitian discrete hypergroup. 
We are going to see several examples of non-distance-regular graphs on which a random walk produces a hermitian discrete hypergroup. 
It is a (graph theoretical) problem to determine pairs of a graph $X = (V, E)$ and a base point $v_0 \in V$ 
which make $(\R (X), \circ_{v_0})$ a hermitian discrete hypergroup. 


\subsection{Fundamental observations} \label{SecNDRGen}

For simplicity of arguments, we set some jargons for graphs. 

\begin{definition} \label{DefHGP}
Let $X = (V, E)$ be a connected graph whose vertices all have finite degrees and $v_0 \in V$. 
\begin{itemize}
\item[$\mathrm{(a)}$] The given graph $X$ is said to satisfy the \textit{self-centered condition} if $X$ is either infinite or self-centered. 
\item[$\mathrm{(b)}$] The pair $(X, v_0)$ is said to be \textit{hypergroup productive} 
if $\R(X)$ becomes a hermitian discrete hypergroup with respect to the convolution $\circ_{v_0}$. 
The graph $X$ is said to be \textit{hypergroup productive} if $(X, v_0)$ is hypergroup productive pair for any $v_0 \in V$. 
\end{itemize}
\end{definition}

It immediately follows from Proposition \ref{PropWellDefined} that the convolution $\circ_{v_0}$ on $\R(X)$ is well-defined 
if and only if $X$ satisfies the self-centered condition. 
We proved in Section \ref{SecDRGWil} that every distance-regular graph is a hypergroup productive graph. 

When $X = (V, E)$ is a distance-regular graph, 
a hypergroup structure of $(\R(X), \circ_{v_0})$ is independent of a choice of the base point $v_0$. 
On the other hand, a hypergroup structure of $(\R(X), \circ_{v_0})$ sometimes depends on a choice of the base point $v_0$ when $X$ is not distance-regular (see Section \ref{SecHGPGraph}). 
The following proposition shows a sufficient condition for that two choices of the base point give the same convolution on $\R(X)$. 

\begin{proposition} \label{PropVerTrans}
Let $X = (V, E)$ be a graph satisfying the self-centered condition and $v_0$, $v_1 \in V$. 
Suppose that there exists an automorphism $\varphi$ of $X$ such that $\varphi (v_0) = v_1$. 
Then, $\circ_{v_0} = \circ_{v_1}$, that is, $S \circ_{v_0} S' = S \circ_{v_1} S'$ holds for any $S$, $S' \in \Complex \R(X)$. 
In particular, if $X$ is vertex-transitive, then $\circ_{v_0} = \circ_{v_1}$ holds for any $v_0$, $v_1 \in V$. 
\end{proposition}

\begin{proof}
Since the convolutions $\circ_{v_0}$ and $\circ_{v_1}$ are bilinear, 
it suffices to show that $R_i \circ_{v_0} R_j = R_i \circ_{v_1} R_j$ for every $i$, $j \in I$. 

Let $R_i \circ_{v_0} R_j = \sum_{k \in I} P_{i,j}^{k} R_k$ and $R_i \circ_{v_1} R_j = \sum_{k \in I} Q_{i,j}^{k} R_k$ for given $i$, $j \in I$. 
Since the automorphism $\varphi$ of $X$ preserves the distance on $X$, 
that is, $\dist{\varphi(v)}{\varphi(w)} = \dist{v}{w}$ holds for any $v$, $w \in V$, we have 
\begin{align*}
Q_{i,j}^{k} &= 
\frac{1}{\abs{\Gamma_{i} (v_1)}} \sum_{v \in \Gamma_{i} (v_1)} \frac{\abs{\Gamma_{j} (v) \cap \Gamma_{k} (v_1)}}{\abs{\Gamma_{j} (v)}} \\
&= \frac{1}{\abs{\Gamma_{i} (\varphi(v_0))}} 
     \sum_{w \in \Gamma_{i} (v_0)} \frac{\abs{\Gamma_{j} (\varphi(w)) \cap \Gamma_{k} (\varphi(v_0))}}{\abs{\Gamma_{j} (\varphi(w))}} \\ 
&= \frac{1}{\abs{\Gamma_{i} (v_0)}} \sum_{w \in \Gamma_{i} (v_0)} \frac{\abs{\Gamma_{j} (w) \cap \Gamma_{k} (v_0)}}{\abs{\Gamma_{j} (w)}} 
  = P_{i,j}^{k}. 
\end{align*}
\end{proof}

It immediately follows from Proposition \ref{PropVerTrans} 
that the hypergroup productivity is preserved by an automorphism in the sense of the following. 

\begin{corollary} \label{CorVerTrans}
Let $X = (V, E)$ be a graph satisfying the self-centered condition and $v_0$, $v_1 \in V$. 
Suppose that there exists an automorphism $\varphi$ of $X$ such that $\varphi(v_0) = v_1$. 
Then, $(X, v_0)$ is a hypergroup productive pair if and only if so is $(X, v_1)$. 
\end{corollary}

A Cayley graph is vertex-transitive, 
so it turns to a hypergroup productive graph if it admits a hypergroup productive pair. 
Moreover, in our context, at most one hypergroup structure can be introduced to the canonical partition of a Cayley graph. 

To show that a given pair $(X, v_0)$ is hypergroup productive, where $X$ is a graph satisfying the self-centered condition and $v_0 \in V$, 
we have to show the commutativity and the associativity of the convolution $\circ_{v_0}$, as we saw in Section \ref{SecDRGWil}. 
In particular, 
the associativity $(R_h \circ_{v_0} R_i) \circ_{v_0} R_j = R_h \circ_{v_0} (R_i \circ_{v_0} R_j)$ should be certificated for every $h$, $i$, $j \in I$, 
but it can be reduced to an easier case. 

\begin{proposition} \label{PropAssCom}
Let $X = (V, E)$ be a graph satisfying the self-centered condition and $v_0 \in V$. 
Assume that the convolution $\circ_{v_0}$ satisfies the following two identities for any $i$, $j \in I$: 
\begin{gather*}
R_i \circ_{v_0} R_j = R_j \circ_{v_0} R_i, \\
(R_1 \circ_{v_0} R_i) \circ_{v_0} R_j = R_1 \circ_{v_0} (R_i \circ_{v_0} R_j). 
\end{gather*}
Then, $\circ_{v_0}$ is commutative and associative on $\Complex \R(X)$, so that $(X, v_0)$ is a hypergroup productive pair. 
\end{proposition}

\begin{proof}
In this proof, we simply write $\circ$ instead of $\circ_{v_0}$. 

Since the convolution $\circ$ is commutative on $\R(X) \times \R(X)$ and bilinear on $\Complex \R(X) \times \Complex \R(X)$, 
the commutativity extends to whole $\Complex \R(X) \times \Complex \R(X)$. 

We will check the associativity of $\circ$. 
If the given convolution $\circ$ is associative on $\R(X)$, 
that is, $(R_h \circ R_i) \circ R_j = R_h \circ (R_i \circ R_j)$ holds for any $h$, $i$, $j \in I$, 
then the associativity of $\circ$ extends to whole $\Complex \R(X)$ from the bilinearity of $\circ$. 
Hence what we have to show is that 
\begin{equation} \label{eqAssBasis}
(R_h \circ R_i) \circ R_j = R_h \circ (R_i \circ R_j)
\end{equation}
for any $h$, $i$, $j \in I$. 

When $h = 0$, \eqref{eqAssBasis} is immediately obtained from the fact that $R_0$ is the neutral element of $\Complex \R(X)$. 
Thus, by our assumption, we have the conclusion when $\diam{X} = 1$. 
We assume that $\diam{X} \geq 2$. 
To show \eqref{eqAssBasis} by induction on $h$, suppose that $h \in I$ with $h \geq 2$ and the claim is true in the cases up to $h-1$. 
By \eqref{eqProbmeas3}, we find that 
\begin{equation*}
R_h = \frac{1}{P_{1,h-1}^{h}} \left( R_1 \circ R_{h-1} - P_{1,h-1}^{h-2} R_{h-2} - P_{1,h-1}^{h-1} R_{h-1} \right). 
\end{equation*}
(Note that it follows that $P_{1,h-1}^{h} > 0$ from the connectivity of $X$.)
By using this identity, we get 
\begin{align*}
&(R_h \circ R_i) \circ R_j \\
=& \frac{1}{P_{1,h-1}^{h}} \left[ \left( (R_1 \circ R_{h-1}) \circ R_i \right) \circ R_j \right] \\
&- \frac{P_{1,h-1}^{h-2}}{P_{1,h-1}^{h}} \left[ (R_{h-2} \circ R_i) \circ R_j \right] 
- \frac{P_{1,h-1}^{h-1}}{P_{1,h-1}^{h}} \left[ (R_{h-1} \circ R_i) \circ R_j \right].
\end{align*}
The induction hypothesis gives that 
\begin{gather}
(R_{h-1} \circ R_i) \circ R_j = R_{h-1} \circ (R_i \circ R_j), \label{eqAss1} \\
(R_{h-2} \circ R_i) \circ R_j = R_{h-2} \circ (R_i \circ R_j), \label{eqAss2}
\end{gather}
and the induction basis gives that 
\begin{equation*}
(R_1 \circ R_{h-1}) \circ R_i = R_1 \circ (R_{h-1} \circ R_i).
\end{equation*}
Therefore, appealing to the induction basis, the induction hypothesis and the bilinearity of $\circ$, we have that 
\begin{align}
& \left( (R_1 \circ R_{h-1}) \circ R_i \right) \circ R_j \notag \\
=& \left( R_1 \circ (R_{h-1}\circ R_i) \right) \circ R_j \notag \\
=& \sum_{k=0}^{h+i-1} P_{h-1, i}^{k} \left( (R_1 \circ R_k) \circ R_j \right) \notag \\
=& \sum_{k=0}^{h+i-1} P_{h-1, i}^{k} \left( R_1 \circ (R_k \circ R_j) \right) \notag \\
=& R_1 \circ \left( (R_{h-1} \circ R_i) \circ R_j \right) \notag \\
=& R_1 \circ \left( R_{h-1} \circ (R_i \circ R_j) \right) \notag \\
=& \sum_{k=0}^{i+j} P_{i,j}^{k} \left( R_1 \circ (R_{h-1} \circ R_k) \right) \notag \\
=& \sum_{k=0}^{i+j} P_{i,j}^{k} \left( (R_1 \circ R_{h-1}) \circ R_k \right) \notag \\
=& (R_1 \circ R_{h-1}) \circ (R_i \circ R_j). \label{eqAss3}
\end{align}


These identities \eqref{eqAss1}, \eqref{eqAss2} and \eqref{eqAss3} make the end of the proof with 
\begin{align*}
&\frac{1}{P_{1,h-1}^{h}} \left[ \left( (R_1 \circ R_{h-1}) \circ R_i \right) \circ R_j \right] \\
&- \frac{P_{1,h-1}^{h-2}}{P_{1,h-1}^{h}} \left[ (R_{h-2} \circ R_i) \circ R_j \right] 
- \frac{P_{1,h-1}^{h-1}}{P_{1,h-1}^{h}} \left[ (R_{h-1} \circ R_i) \circ R_j \right] \\
=& \frac{1}{P_{1,h-1}^{h}} \left[ (R_1 \circ R_{h-1}) \circ (R_i \circ R_j) \right] \\
&- \frac{P_{1,h-1}^{h-2}}{P_{1,h-1}^{h}} \left[ R_{h-2} \circ (R_i \circ R_j) \right] 
- \frac{P_{1,h-1}^{h-1}}{P_{1,h-1}^{h}} \left[ R_{h-1} \circ (R_i \circ R_j) \right] \\
=& \left[ \frac{1}{P_{1,h-1}^{h}} \left( R_1 \circ R_{h-1} - P_{1,h-1}^{h-2} R_{h-2} - P_{1,h-1}^{h-1} R_{h-1} \right) \right] \circ (R_i \circ R_j) \\
=& R_h \circ (R_i \circ R_j). 
\end{align*}
\end{proof}


\subsection{Examples of hypergroup productive graphs} \label{SecHGPGraph}

We see some examples of non-distance-regular hypergroup productive graphs in this section. 
One can check the associativity by direct calculations for each case, 
so we shall only give the structure identities and omit the proof of associativity. 

\begin{itemize}
\item[$\mathrm{(i)}$] Prisms
\end{itemize}

The prism graphs are the simplest examples of non-distance-regular graphs, 
on which a random walk produces a hermitian discrete hypergroup. 
Consider the $n$-gonal prism for $n \geq 3$ and let $V_n$ denote the set of its vertices and $E_n$ the set of its edges. 
The graph $\mathcal{P}_n = (V_n, E_n)$ is distance-regular if and only if $n = 4$ ($\mathcal{P}_4 \cong \mathcal{S}_6$), 
but any $n \geq 3$ allows $\mathcal{P}_n$ to produce a hermitian discrete hypergroup. 

The $n$-gonal prism graph $\mathcal{P}_n$ can be realized as a Cayley graph; 
$\mathcal{P}_n = \Cayley{\Integer / n \Integer \oplus \Integer / 2 \Integer}{\set{(\overline{\pm 1}, \overline{0}), (\overline{0}, \overline{1})}}$, 
where $\overline{a}$ denotes the residue class of $a \in \Integer$. 
Hence, by Proposition \ref{PropVerTrans}, 
the hypergroup structure $\left( \R(\mathcal{P}_n), \circ_{v_0} \right)$ is independent of a choice of the base point $v_0$. 

The structure identities of $\R(\mathcal{P}_n)$ are computed to be, if $n = 2m+1$ with $m \geq 2$, 
\begin{multline} \notag 
R_1 \circ R_j = R_j \circ R_1 \\
= \frac{3-\delta_{j,1}+\delta_{j,m+1}}{6} R_{j-1} + \frac{\delta_{j,m}}{6} R_j + \frac{3+\delta_{j,1}-\delta_{j,m}-\delta_{j,m+1}}{6} R_{j+1} \\
(1 \leq j \leq m+1), 
\end{multline}
\begin{multline} \label{eqPrismOdd2}
R_i \circ R_j = R_j \circ R_i \\
= \frac{3-\delta_{i,j}}{8} R_{\abs{i-j}} + \frac{1+\delta_{i,j}}{8} R_{\abs{i-j}+2} + \frac{1}{8} R_{i+j-2} + \frac{3}{8} R_{i+j} \\
(2 \leq i, j \leq m-1,\, i+j \leq m), 
\end{multline}
\begin{multline} \label{eqPrismOdd3}
R_i \circ R_j = R_j \circ R_i \\
= \frac{3-\delta_{i,j}}{8} R_{\abs{i-j}} + \frac{1+\delta_{i,j}}{8} R_{\abs{i-j}+2} + \frac{1}{8} R_{m-1} + \frac{1}{8} R_m + \frac{1}{4} R_{m+1} \\
(2 \leq i, j \leq m-1,\, m+1 \leq i+j \leq m+2), 
\end{multline}
\begin{multline} \label{eqPrismOdd4}
R_i \circ R_j = R_j \circ R_i \\
= \frac{3-\delta_{i,j}}{8} R_{\abs{i-j}} + \frac{1+\delta_{i,j}}{8} R_{\abs{i-j}+2} + \frac{1}{8} R_{2m-i-j+1} + \frac{3}{8} R_{2m-i-j+3} \\
(2 \leq i, j \leq m-1,\, i+j \geq m+3), 
\end{multline}
\begin{multline} \label{eqPrismOdd5}
R_i \circ R_m = R_m \circ R_i \\
= \frac{3}{8} R_{m-i} + \frac{1}{8} R_{m-i+1} + \frac{1+\delta_{i,2}}{8} R_{m-i+2} + \frac{3-\delta_{i,2}}{8} R_{m-i+3} \\
(2 \leq i \leq m-1), 
\end{multline}
\begin{equation} \notag 
R_m \circ R_m = \frac{1}{4} R_0 + \frac{1}{8} R_1 + \frac{2 + \delta_{m,2}}{8} R_2 + \frac{3 - \delta_{m,2}}{8} R_3, 
\end{equation}
\begin{multline} \notag 
R_i \circ R_{m+1} = R_{m+1} \circ R_i = \frac{3+\delta_{i,1}}{6} R_{m-i+1} + \frac{3-\delta_{i,1}}{6} R_{m-i+2} \\
(1 \leq i \leq m+1), 
\end{multline}
and to be, if $n = 2m$ with $m \geq 2$, 
\begin{multline} \notag 
R_1 \circ R_j = R_j \circ R_1 = \\
\frac{3 - \delta_{j,1} + \delta_{j,m} + 3 \delta_{j,m+1}}{6} R_{j-1} + \frac{3 + \delta_{j,1} - \delta_{j,m} - 3 \delta_{j,m+1}}{6} R_{j+1} \\
(1 \leq j \leq m+1), 
\end{multline}
\begin{multline} \label{eqPrismEven2}
R_i \circ R_j = R_j \circ R_i = \\
\frac{3 - \delta_{i,j}}{8} R_{\abs{i-j}} + \frac{1 + \delta_{i,j}}{8} R_{\abs{i-j}+2} + 
 \frac{1 + \delta_{i+j,m+1}}{8} R_{i+j-2} + \frac{3 - \delta_{i+j,m+1}}{8} R_{i+j} \\
(2 \leq i, j \leq m-1,\, i+j \leq m+1), 
\end{multline}
\begin{multline} \label{eqPrismEven3}
R_i \circ R_j = R_j \circ R_i = \\
\frac{3 - \delta_{i,j}}{8} R_{\abs{i-j}} + \frac{1 + \delta_{i,j}}{8} R_{\abs{i-j}+2} + 
 \frac{1}{8} R_{2m-i-j} + \frac{3}{8} R_{2m-i-j+2} \\
(2 \leq i, j \leq m-1,\, i+j \geq m+2), 
\end{multline}
\begin{equation} \label{eqPrismEven4}
R_i \circ R_m = R_m \circ R_i = 
\frac{1}{2} R_{m-i} + \frac{1}{2} R_{m-i+2} \quad 
(2 \leq i \leq m-1), 
\end{equation}
\begin{equation} \notag 
R_m \circ R_m = \frac{1}{3} R_0 + \frac{2}{3} R_2, 
\end{equation}
\begin{equation} \notag 
R_i \circ R_{m+1} = R_{m+1} \circ R_i = 
R_{m-i+1} \quad 
(1 \leq i \leq m+1). 
\end{equation}
The identities for $R_i \circ R_j$ with $2 \leq i \leq m-1$ or with $2 \leq j \leq m-1$ in the above 
(i.e.\ \eqref{eqPrismOdd2} -- \eqref{eqPrismOdd5} and \eqref{eqPrismEven2} -- \eqref{eqPrismEven4}) should be omitted when $m = 2$. 
If $n = 3$, the structure identities of $\R(\mathcal{P}_3)$, given as follows, are slightly different from the above ones; 
\begin{gather*}
R_1 \circ R_1 = \frac{1}{3} R_0 + \frac{2}{9} R_1 + \frac{4}{9} R_2, \\
R_1 \circ R_2 = R_2 \circ R_1 = \frac{2}{3} R_1 + \frac{1}{3} R_2, \\
R_2 \circ R_2 = \frac{1}{2} R_0 + \frac{1}{2} R_1. 
\end{gather*}

We note that there can be constructed a finite hermitian discrete hypergroup of arbitrarily large order from a prism graph. 

\begin{itemize}
\item[$\mathrm{(ii)}$] Finite regular graphs producing two types of hermitian discrete hypergroups
\end{itemize}

Thus far every example of hypergroup productive graphs have induced a single structure of hermitian discrete hypergroups 
in its canonical partition, 
whereas there exist hypergroup productive graphs which produce two (or more) structures of hermitian discrete hypergroups. 
Here, we see two such graphs, which are drawn as in Figures \ref{FigHeptagon} and \ref{FigLine3Prism}. 
The latter one can be realized as the line graph of the triangular prism $\mathcal{P}_3$. 
(The definition of the line graph refers to \cite{Big} etc.)
Both of two graphs are of diameter two, and neither one is found to be vertex-transitive from Proposition \ref{PropVerTrans}. 

Let the graph drawn in Figure \ref{FigHeptagon} be denoted by $X_1$ and the other one $X_2$. 
The structure identities of $\R(X_1)$ are given by 
\begin{gather*}
R_1 \circ_{v_0} R_1 = \frac{1}{4} R_0 + \frac{1}{4} R_1 + \frac{1}{2} R_2, \\
R_1 \circ_{v_0} R_2 = R_2 \circ_{v_0} R_1 = R_1, \\
R_2 \circ_{v_0} R_2 = \frac{1}{2} R_0 + \frac{1}{2} R_2 
\end{gather*}
if the base point $v_0$ is chosen from filled vertices in Figure \ref{FigHeptagon} and 
\begin{gather*}
R_1 \circ_{v_0} R_1 = \frac{1}{4} R_0 + \frac{3}{8} R_1 + \frac{3}{8} R_2, \\
R_1 \circ_{v_0} R_2 = R_2 \circ_{v_0} R_1 = \frac{3}{4} R_1 + \frac{1}{4} R_2, \\
R_2 \circ_{v_0} R_2 = \frac{1}{2} R_0 + \frac{1}{2} R_1 
\end{gather*}
if the base point $v_0$ is chosen from blank vertices in the same figure. 
On the other hand, the structure identities of $\R(X_2)$ are computed to be 
\begin{gather*}
R_1 \circ_{v_0} R_1 = \frac{1}{4} R_0 + \frac{3}{8} R_1 + \frac{3}{8} R_2, \\
R_1 \circ_{v_0} R_2 = R_2 \circ_{v_0} R_1 = \frac{3}{8} R_1 + \frac{5}{8} R_2, \\
R_2 \circ_{v_0} R_2 = \frac{1}{4} R_0 + \frac{5}{8} R_1 + \frac{1}{8} R_2
\end{gather*}
if the base point $v_0$ chosen from filled vertices in Figure \ref{FigLine3Prism} and 
\begin{gather*}
R_1 \circ_{v_0} R_1 = \frac{1}{4} R_0 + \frac{1}{4} R_1 + \frac{1}{2} R_2, \\
R_1 \circ_{v_0} R_2 = R_2 \circ_{v_0} R_1 = \frac{1}{2} R_1 + \frac{1}{2} R_2, \\
R_2 \circ_{v_0} R_2 = \frac{1}{4} R_0 + \frac{1}{2} R_1 + \frac{1}{4} R_2
\end{gather*}
if the base point $v_0$ chosen from blank vertices in the same figure.  

\begin{figure}[t]
\begin{minipage}{0.5\hsize}
\centering
\includegraphics{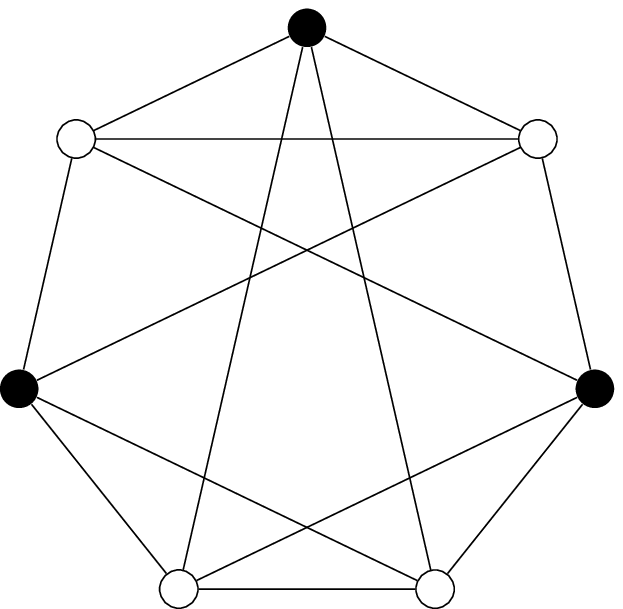}
\caption{A $4$-regular graph which produces two hypergroup structures}
\label{FigHeptagon}
\end{minipage}
\begin{minipage}{0.5\hsize}
\centering
\includegraphics{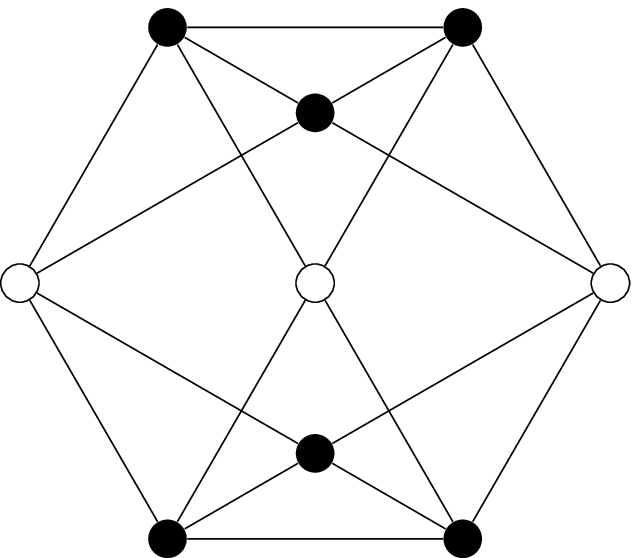}
\caption{The line graph of $\mathcal{P}_3$}
\label{FigLine3Prism}
\end{minipage}
\end{figure}

\begin{itemize}
\item[$\mathrm{(iii)}$] Complete bipartite graphs
\end{itemize}

Let $m$, $n \in \Natural$ with $m$, $n \geq 2$. 
The \textit{complete bipartite graph} $K_{m,n} = (V_{m,n}, E_{m,n})$ is defined as follows: 
\begin{itemize}
\item $V_{m,n} = \set{u_1, u_2, \cdots, u_m, w_1, w_2, \cdots, w_n}$ ($\abs{V_{m,n}} = m + n$). 
\item $E_{m,n} = \setcond{\set{u_{\alpha}, v_{\beta}}}{1 \leq \alpha \leq m, 1 \leq \beta \leq n}$. 
\end{itemize}

The complete bipartite graph $K_{m,n}$ is of diameter two and is regular if and only if $m = n$. 
We here meet an intriguing example, which is a non-regular hypergroup productive graph. 

We can compute the structure identities of $\R(K_{m,n})$ to be 
\begin{gather*}
R_1 \circ_{v_0} R_1 = \frac{1}{m} R_0 + \frac{m-1}{m} R_2, \\
R_1 \circ_{v_0} R_2 = R_2 \circ_{v_0} R_1 = R_1, \\
R_2 \circ_{v_0} R_2 = \frac{1}{m-1} R_0 + \frac{m-2}{m-1} R_2
\end{gather*}
if the base point $v_0$ is chosen from $u_{\alpha}$'s and to be 
\begin{gather*}
R_1 \circ_{v_0} R_1 = \frac{1}{n} R_0 + \frac{n-1}{n} R_2, \\
R_1 \circ_{v_0} R_2 = R_2 \circ_{v_0} R_1 = R_1, \\
R_2 \circ_{v_0} R_2 = \frac{1}{n-1} R_0 + \frac{n-2}{n-1} R_2
\end{gather*}
if the base point $v_0$ is chosen from $v_{\beta}$'s. 
Needless to say, these identities completely coincide when $m = n$. 

\begin{itemize}
\item[$\mathrm{(iv)}$] Infinite ladder graph
\end{itemize}

Here will be introduced an example of infinite hypergroup productive graphs, which can be drawn like a ladder as in Figure \ref{FigLadder}. 
More precisely, we consider the Cayley graph 
$\L = \Cayley{\Integer \oplus \left( \Integer / 2 \Integer \right)}{\set{(\pm 1, \overline{0}), (0, \overline{1})}}$ in this part. 
(As mentioned in Theorem \ref{ThmInfty}, 
$\overline{0}$ and $\overline{1}$ denote the residue classes of $0$ and $1$ modulo $2$, respectively.) 
Since Cayley graphs are vertex-transitive, 
Proposition \ref{PropVerTrans} allows one to assume that $v_0 = (0, \overline{0})$ is the base point. 
The structure identities of $\R(\L)$ can be computed to be 
\begin{gather*}
R_1 \circ R_1 = \frac{1}{3} R_0 + \frac{2}{3} R_2, \\
R_1 \circ R_i = R_i \circ R_1 = \frac{1}{2} R_{i-1} + \frac{1}{2} R_{i+1} \quad (i \geq 2), \\
R_i \circ R_i = \frac{1}{4} R_0 + \frac{1}{4} R_2 + \frac{1}{8} R_{2i-2} + \frac{3}{8} R_{2i} \quad (i \geq 2), \\
R_i \circ R_j = \frac{3}{8} R_{\abs{i-j}} + \frac{1}{8} R_{\abs{i-j}+2} + \frac{1}{8} R_{i+j-2} + \frac{3}{8} R_{i+j} \quad (i, j \geq 2,\, i \neq j). 
\end{gather*}
To prove the claim $\mathrm{(b)}$ of Theorem \ref{ThmInfty}, the associativity remains to be checked. 
It can be shown by elementary calculations (and Proposition \ref{PropAssCom}). 

\begin{figure}[t]
\centering
\includegraphics{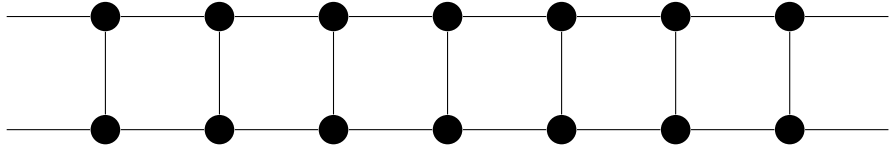}
\caption{Infinite ladder graph}
\label{FigLadder}
\end{figure}

\begin{remark}
One can find that the Cayley graph 
$\Cayley{\Integer \oplus \left( \Integer / n \Integer \right)}{\set{(\pm 1, \overline{0}), (0, \overline{\pm 1})}}$ admits 
no hypergroup productive pairs when $n \geq 3$. 
In addition, it turns out that the Cayley graph $\Cayley{\Integer \oplus \Integer}{\set{(\pm 1, 0), (0, \pm 1)}}$, 
which is actually expected to be hypergroup productive, 
also admits no hypergroup productive pairs. 
(This graph can be drawn as the square lattice in the Euclidean plane.)
There is no known examples of graphs that admit a hypergroup productive pair but fail to be hypergroup productive graphs. 
\end{remark}
 
\begin{remark}
Hypergroups derived from a random walk on an infinite graph can be realized as \textit{polynomial hypergroups} 
since $\supp{R_1 \circ R_i} \subset \set{i-1, i, i+1}$. 

In particular, the hypergroup derived from a random walk on $\T_2$ coincides 
with the hypergroup which can be constructed by the Chebyshev polynomials of the first kind, 
so that the hypergroup $\R(\T_2)$ is called the Chebyshev hypergroup of the first kind. 
Tsurii \cite{Tsu15} elucidated several properties of the Chebyshev hypergroup of the first kind. 
(He also investigate the Chebyshev hypergroup of the second kind in \cite{Tsu15}.)

For details of the polynomial hypergroups, see \cite[Chapter 3]{Blo-Hey} or \cite{Las83} for example. 
\end{remark}

\begin{acknowledgment}
The authors express their gratitude to Prof.\ Satoshi Kawakami, Prof.\ Kohji Matsumoto, Prof.\ Tatsuya Tsurii, Prof.\ Shigeru Yamagami 
and Mr.\ Ippei Mimura for helpful comments. 
\end{acknowledgment}






\end{document}